\documentclass[11pt,a4paper]{amsart}

\usepackage{amsmath,amsfonts,amssymb,amsthm}
\usepackage{txfonts}
\usepackage{latexsym,bm,graphicx}
\usepackage{mathrsfs}
\usepackage{color}
\usepackage{calligra}

\title[Quantitative gradient estimates for harmonic maps]{Quantitative gradient estimates for harmonic maps into  singular spaces}

\author{Hui-Chun Zhang}
\address{Department of Mathematics\\  Sun Yat-sen University\\ Guangzhou 510275\\ \newline E-mail address: zhanghc3@mail.sysu.edu.cn}
\author{Xiao Zhong}
\address{Department of Mathematics and Statistics\\ University of Helsinki\\ FI-00014 University of Helsinki \\ \newline E-mail address: xiao.x.zhong@helsinki.fi}
\author{Xi-Ping Zhu}
\address{Department of Mathematics\\  Sun Yat-sen University\\ Guangzhou 510275\\ \newline E-mail address: stszxp@mail.sysu.edu.cn}
\newtheorem{thm}{Theorem}[section]
\newtheorem{prop}[thm]{Proposition}
\newtheorem{lem}[thm]{Lemma}

\newtheorem{cor}[thm]{Corollary}

\theoremstyle{definition}
\theoremstyle{remark}

\newtheorem{defn}[thm]{Definition}
\newtheorem{rem}[thm]{Remark}

\numberwithin{equation}{section}

\newcommand{\ls}{\leqslant}
\newcommand{\gs}{\geqslant}

\newcommand{\ip}[2]{\left<{#1},{#2}\right>}
\newcommand{\dv}{{\rm d}v_g}

\newcommand{\R}{\mathbb{R}}

\begin{document}


\begin{abstract} In this paper, we will show the Yau's gradient estimate for harmonic maps  into a metric space $(X,d_X)$ with curvature bounded above by a constant $\kappa$, $\kappa\gs0$, in the sense of Alexandrov. As a direct application, it  gives some Liouville theorems for such harmonic maps. This extends the works of S. Y. Cheng \cite{cheng80} and H. I. Choi \cite{choi82} to harmonic maps into singular spaces.
\end{abstract}

\dedicatory{Dedicated to Professor Yang Lo on the Occasion of his {\rm80}th Birthday}



\maketitle
 \tableofcontents
 \setcounter{tocdepth}{1}

\section{Introduction}

Let $M, N$ be two smooth Riemannian manifolds. There is a natural concept of energy functional for  $C^1$-maps  between $M$ and $ N$.  The local minimizers (or more general critical points) of such an energy functional are called harmonic maps. Regularity of harmonic maps is an important topic in the field of geometric analysis. If $\dim M=2$, the regularity of energy minimizing harmonic maps was established  by C. Morrey \cite{morrey48}. If $\dim M\gs3$, a beautiful regularity theory was established by  Schoen-Uhlenbeck \cite{schoen-u82,schoen-u84}, and in a somewhat different context, by Giaquinta-Giusti \cite{giaquinta-g82,giaquinta-g84} (and by \cite{hkw77} when the image of the map is contained in a convex ball of $N$).

In 1975, Yau established  a seminal interior gradient estimate for harmonic functions on Riemanian manifolds with Ricci curvature bounded below. In 1980, Cheng \cite{cheng80} generalized the Yau's gradient estimate to harmonic maps.
\begin{thm}{\rm (Cheng \cite{cheng80})} Let $M,N$ be complete Riemannian manifolds such that $M$ has Ricci curvature $Ric_M\gs -K$, $K\gs0$, and that $N$ is simply-connected and is having non-positive sectional curvature. Let $f: M\to N$ be a harmonic map. Assume that $f\big(B_a(x_0)\big)\subset B_b(y_0)$ for some $x_0\in M$, $y_0\in N$ and some $a,b>0$. Then we have
\begin{equation}
\sup_{B_{a/2}(x_0)}|\nabla f|^2\ls C_n\cdot\frac{b^4}{a^4}\cdot  \max\big\{\frac{Ka^4}{b^2},\frac{a^2(1+Ka^2)}{b^2},\frac{a^2}{b^2}\big\},
\end{equation}
where $C_n$ is a constant depending only on $n=\dim(M)$.
\end{thm}
In particular, when $K=0$, this implies a Liouville theorem: if the  $f$ is bounded, then it is a constant map. Choi \cite{choi82} further extended Cheng's work \cite{cheng80} as following.
\begin{thm}{\rm (Choi \cite{choi82} ).} Let $M,N$ be complete Riemannian manifolds such that $M$ has Ricci curvature $Ric_M\gs -K$, $K\gs0$, and that $N$ has sectional curvature $\sec_N\ls \kappa$, $\kappa>0$. Let $f:M\to N$ be a harmonic map. Assume that $f(M)\subset B_b(y_0)$ lies inside the cut locus of $y_0\in N$ and some $b<\pi/(2\sqrt\kappa)$.  Then $|\nabla f|$ is bounded by a constant depending only on $n,K,\kappa$ and $b$. If, furthermore, $K=0$, then $f$ is a constant map.
\end{thm}
It is well known from \cite{hkw77,jk83} that the radius $b<\pi/(2\sqrt\kappa)$ is sharp. Without the restriction that the image of $u$ is contained in a ball with radius $\pi/(2\sqrt \kappa)$, a harmonic map  might not be even continuous.

\subsection{Yau's gradient estimates for harmonic maps into metric spaces}$\ $

 The purpose of this paper is to extend the Yau's gradient estimate to harmonic maps into singular metric spaces.

In the seminal work of M. Gromov and R. Schoen \cite{gromov-schoen92}, they initiated to study   harmonic maps into singular spaces. A general theory of harmonic maps between singular spaces was developed by Korevaar-Schoen \cite{korevaar-schoen93}, Jost \cite{jost94,jost97} and Lin \cite{lin97}, independently.

If $u$ is a map from a domain $\Omega\subset M$ of Riemannian manifold to an arbitrarily metric space  $(X,d_X)$, which is unnecessary to  be embedded  into a Euclidean space,   N. Korevaar and R. Schoen \cite{korevaar-schoen93} introduced an intrinsic approach to generalize the concept of  the energy of $u$.
Given   a map $u\in L^2(\Omega,X),$
 for each $\epsilon>0$, the \emph{approximating energy} $E^u_{\epsilon}$ is defined as a functional on $C_0(\Omega)$:
$$E^u_{\epsilon}(\phi):=\int_\Omega\phi(x) e^u_{\epsilon}(x)\dv(x)$$
where $\phi\in C_0(\Omega)$, the space of continuous functions compactly supported on $\Omega$,  and $e^u_{\epsilon}$ is \emph{approximating energy density} defined by
$$e^u_{\epsilon}(x):=\frac{n(n+2)}{\omega_{n-1}\cdot\epsilon^n}\int_{B_\epsilon(x)\cap\Omega}\frac{d^2_X\big(u(x),u(y)\big)}{\epsilon^2}\dv(y),$$
 where $\omega_{n-1}$ is the volume of $(n-1)$-sphere $\mathbb S^{n-1}$ with the standard metric.
In \cite{korevaar-schoen93},  Korevaar-Schoen proved that
$$\lim_{\epsilon\to0^+} E^u_{\epsilon}(\phi)=E^u(\phi)$$
for some positive functional $E^u(\phi)$ on $C_0(\Omega)$. The limit functional $E^u$ is called the energy (functional) of $u$. By Riesz representation theorem, the nonnegative functional $E^u$ is a Radon measure on $\Omega$. Moreover, Korevaar-Schoen in \cite{korevaar-schoen93} proved that the measure is absolutely continuous respect to the Riemannian volume ${\rm vol}_g$. Denote $e_u:=\frac{dE^u}{d{\rm vol}_g}$, the energy density  of $u$.
For a smooth map $f$ between two smooth Riemannian manifolds, we have $e_f={\rm const}\cdot |\nabla f|^2$.

 The (local) minimizing maps, in the sense of calculus of variations, of such an energy functional $E^u$ are called \emph{harmonic maps}.

 If $(X,d_X)$ is a locally compact Riemannian simplicial complex with (globally) non-positive curvature in the sense of Alexandrov,   Gromov-Schoen \cite{gromov-schoen92} established  the local Lipschitz regularity for harmonic maps from $\Omega$ to $X$. Korevaar-Schoen  \cite{korevaar-schoen93} extended to the case where $X$ is a general $CAT(0)$-space, a metric space with non-positive curvature in the sense of Alexandrov. A further extension was given by Serbinowski \cite{serbinowski95}. Let us put these regularity results together as follows.

  \begin{thm}\label{thm1.3}
  {\rm (Korevaar-Schoen \cite[Theorem 2.4.6]{korevaar-schoen93}, Serbinowski  \cite[Corollary 2.18]{serbinowski95}).}
Let $\Omega\subset M$  be  a   bounded domain (with smooth boundary) of a  Riemannian manifold $(M,g)$ and let $(X,d_X)$  be a $CAT(\kappa)$-space for some $\kappa\gs 0$.
 Suppose that $u:\Omega\to X$ is a harmonic map. Assume that the image of $u$ is contained in a ball with radius $\rho<\pi/(2\sqrt \kappa)$. Here and in the sequel, if $\kappa=0$, we always understand $\pi/(2\sqrt\kappa)=+\infty.$
 Then $u$ is locally Lipschitz continuous in $\Omega$. Moreover, for any   ball $B_{R}(o)\subset\subset\Omega$,  it holds the following Bernstein-type gradient estimate
 \begin{equation}\label{grad}
\sup_{B_{R/2}(o)}e_u\ls C\fint_{B_{R}(o)} e_u\dv,
\end{equation}
where the    constant $C$ depends on $n=\dim(M)$, $R$, the injectivity radius of $o$, $\pi/(2\sqrt{\kappa})-\rho$,  and the $C^1$-norm of metric coefficients $g$ on $B_{R}(o)$. Here and in the sequel, $\fint_E :=\frac{1}{{\rm vol}_g(E)}\int_E $ denotes the average integral over the measurable set $E$.
\end{thm}
In the last two decades,  many regularity results have been obtained for   (energy minimizing) harmonic maps into or between singular spaces (see, for example, \cite{korevaar-schoen93,jost96,jost97,sturm05,chen95,eells-fug01-book,fug03,dm08,dm10,lin97,lin-wang08,zz17-lip,guo17,bfhmsz17} and so on).

For the case when the domain $\Omega$ has nonnegative sectional curvature and the target $X$ is a $CAT(0)$-simplicial complex, J. Chen \cite{chen95} showed that   the constant $C$ in (\ref{grad}) depends only on $n$.  When the target $X$ is a general $CAT(0)$-space, Jost \cite{jost98} gave an approach to  deduce  an explicit bound of the  constant  in (\ref{grad}) in terms of the sectional  curvature of $M$, $n$ and $R$. Other quantitative Lipschitz estimates of $u$ were also given in \cite{dm08,dm10}.

In \cite[Sect. 6, Page 167]{jost98}, J. Jost proposed an open problem, in the case when the target $X$ is a $CAT(0)$-space, to ask if  the $\sup_{B_{R/2}(o)}e_u$ can be dominated by a constant   depending only on the lower bound for the Ricci curvature of $M$, the dimension of $M$, and the energy of $u$. Furthermore, a natural problem was arisen from the combination of the Jost's problem and the Cheng's work \cite{cheng80} to ask if a Yau-type interior gradient estimate holds  for harmonic map into a $CAT(0)$-space.
  The first   result in this paper answers this problem affirmatively.
 \begin{thm}\label{thm1.4}
 Let $\Omega$ be a bounded domain (with smooth boundary) of an $n$-dimensional Riemannian manifold $(M,g)$ with $Ric_{M}\gs - K$ for some $K\gs0$, and let $(X,d_X)$ be a $CAT(0)$-space.
Suppose that   $u:\Omega\to X$ is a  harmonic map.
Given any  ball $B_R(x_0)$ with   $B_{2R}(x_0) \subset\subset \Omega,$ if $u(B_R(x_0))\subset B_\rho(Q_0)$ for some $Q_0\in X$ and some $\rho>0$, then we have
\begin{equation*}\label{quantitative}
\sup_{B_{R/2}(x_0)}{\rm Lip}u\ls C_{n,\sqrt KR}\cdot\frac{\rho}{R},
\end{equation*}
where ${\rm Lip}u$ is the pointwise Lipschitz constant given by
$${\rm Lip}u(x):=\limsup_{y\to x}\frac{d_X\big(u(x),u(y)\big)}{d(x,y)},$$
and where $d(x,y)$ is the distance with respect to the Riemannian metric $g$ on $M$, and $C_{n,\sqrt KR}$ is a constant depending only on $n$ and $\sqrt KR$.
  \end{thm}
\begin{rem} (1)\indent
It is clear from the definitions of $e_u$ and ${\rm Lip}u$ that  $e_u(x)\ls (n+2){\rm Lip}^2u(x)$ for almost all $x\in \Omega$.

(2) \indent By the fact $\Delta d^2_X\big(u(x),u(x_0) \big)\gs 2e_u\gs0$ (see \cite{jost97} or Lemma \ref{lem2.5}), it is well-known that the $\sup_{x\in B_{R/2}(x_0)}d^2_X\big(u(x),u(x_0)\big)$ can be dominated by   $C_{n,\sqrt KR}R^2\cdot\fint_{B_R(x_0)}e_udv_g$ (see, for example, Eq.(\ref{equ2.7})). So, by choosing $Q_0=u(x_0)$,  Theorem \ref{thm1.4} implies that
\begin{equation}\label{jost-ques}
\sup_{B_{R/2}(x_0)}{\rm Lip}u\ls C_{n,\sqrt KR}\Big(\fint_{B_{R}(x_0)}e_u {\rm d}v_g\Big)^{1/2}.
\end{equation}
It answers  the Jost's problem (\cite[Sect. 6]{jost98}) affirmatively.
\end{rem}
As an immediate application of Theorem \ref{thm1.4}, by letting $R\to\infty$,  we have  the following Liouville theorem (cf. \cite[Theorem 1.4]{sin14} and \cite[Theorem 1.2]{hua-l-x17}).
\begin{cor}\label{cor1.7}
 Let $(M,g)$ be  an $n$-dimensional complete non-compact Riemannian manifold with nonnegative Ricci curvature, and let $(X,d_X)$ be a $CAT(0)$-space.
Let $u:M\to X$ be a harmonic map. If $u$ satisfies sub-linear growth:
$$\liminf_{R\to\infty} \frac{\sup_{y\in B_R(x_0)}d_X\big(u(y),Q_0\big)}{R}=0$$
for some $Q_0\in X$, then $u$ must be a constant map.
\end{cor}
For the case when the target space has curvature $\ls \kappa$ for some $\kappa>0$, we have the following gradient estimate.
\begin{thm}\label{thm1.8}
 Let $\Omega$ be a bounded domain (with smooth boundary) of an $n$-dimensional Riemannian manifold $(M,g)$ with $Ric_{M}\gs - K$ for some $K\gs0$, and let $(X,d_X)$ be a $CAT(\kappa)$-space, $\kappa>0$.
Suppose that   $u:\Omega\to X$ is a  harmonic map with the image $u(\Omega)\subset B_\rho(Q_0)$ for some $Q_0\in X$ and $\rho<\pi/(2\sqrt\kappa)$.
Then we have
\begin{equation*}\label{quantitative-cat1}
\sup_{B_{R/2}(x_0)}{\rm Lip}u\ls\frac{C_{n,\sqrt KR,\pi/(2\sqrt\kappa)-\rho}}{R},
\end{equation*}
where $C_{n,\sqrt KR,\pi/(2\sqrt\kappa)-\rho}$ is a constant depending only on $n,\sqrt KR$ and $\pi/(2\sqrt\kappa)-\rho$.
\end{thm}
This implies the following Liouville theorem, by letting $R\to\infty$.
\begin{cor}\label{cor1.9}
 Let $(M,g)$ be  an $n$-dimensional complete non-compact Riemannian manifold with nonnegative Ricci curvature, and let $(X,d_X)$ be a $CAT(1)$-space.
Let $u:M\to X$ be a harmonic map. If $u(M)\subset B_\rho(Q_0)$ for some $Q_0\in X$ and $\rho<\pi/2$, then $u$ must be a constant map.
\end{cor}
\begin{rem}
If $u(M)\subset B_{\pi/2}(Q_0)$  and if $d^2_X\big(Q_0,u(x)\big)\in L^1(M)$, then the same conclusion, $u$ is a constant map, has been proved recently by B. Freidin and Y. Zhang in \cite{freidin-zhang18}.
\end{rem}

\subsection{A sharp Bochner inequality for the harmonic maps into metric spaces} $\ $

 Cheng's argument in \cite{cheng80} is based on  the classical Bochner formula of Eells and Sampson. That is,
  for a smooth harmonic map $u$ between two Riemannian manifolds $M$ and $N$, it holds:
  \begin{equation}\label{classical-bochner}
  \begin{split}
  \frac 1 2 |\nabla u|^2&=|d\nabla u|^2+Ric_M(\nabla u,\nabla u)-\ip{R^N(u_*e_\alpha,u_*e_\beta)u_*e_\alpha}{u_*e_\beta}\\
  &\gs |\nabla |\nabla u||^2-K|\nabla u|^2-\kappa|\nabla u|^4,
  \end{split}
  \end{equation}
  where the Ricci curvature of $M$ is bounded below by $-K$ and the sectional curvature of $N$ is bounded above by $\kappa$.
  It is clear that the classical Bochner formula relies heavily on the smoothness of the target space $X$ (requiring to have at least second order derivatives).

 For harmonic maps into singular spaces, it is a basic problem to deduce a Bochner type formula.  For the case when the domain $\Omega$ has nonnegative sectional curvature and the target $X$ is a non-positively curved simplicial complex, J. Chen \cite{chen95} used the method in \cite{gromov-schoen92} to show that $e_u$ is a sub-harmonic function on $\Omega$.
In \cite{korevaar-schoen93}, Korevaar-Schoen developed a finite difference technique to prove the following weak form of the  Bochner type inequality:  there exists a constant $C$, depending on the $C^1$-norm of $g$, such that
  $$\int_\Omega e_u\big(\Delta \eta+C|\nabla \eta|+C\eta\big){\rm d}v_g\gs0$$
  for all $\eta\in C^\infty_0(\Omega)$.   Korevaar-Schoen's method in \cite{korevaar-schoen93} has been extended by Serbinowski \cite{serbinowski95} to the case when the target space is of $CAT(\kappa)$ for any $\kappa>0$. Mese \cite{mese01} showed that $\Delta e_u\gs -2\kappa e_u^2$, in the sense of distributions, for a harmonic map from a flat domain to a $CAT(\kappa)$-space.  Recently,  Freidin \cite{freidin16} and Freidin-Zhang \cite{freidin-zhang18} improved the method in \cite{gromov-schoen92} to deduce  the following Bochner type inequality for a harmonic map from a Riemannian manifold into a $CAT(\kappa)$-space:
\begin{equation}\label{pre-bochner}
\frac 1 2 \Delta e_u \gs  - K e_u-  \kappa e_u^2,
\end{equation} in the sense of distributions.

Recalling the arguments of Cheng \cite{cheng80} and Choi \cite{choi82}, the key intergradient is the  positive term $|\nabla |\nabla u||^2$ in the right hand side of (\ref{classical-bochner}).
The Bochner inequality (\ref{pre-bochner})  is not enough to get the Theorem \ref{thm1.4} and Theorem \ref{thm1.8}.
In this paper,  we will  establish  a generalized Bochner inequality keeping such a positive term as follows.
  \begin{thm}\label{bochner}
Let $\Omega$ be a smooth domain of an $n$-dimensional Riemannian manifold $(M,g)$ with $Ric_{M}\gs - K$ for some $K\gs0$, and let $(X,d_X)$ be a $CAT(\kappa)$-space for some $\kappa\gs0$.
Suppose that the map $u:\Omega\to X$ is harmonic and that  its image $u(\Omega)$ is contained in a ball $B_\rho\subset X$ with radius $\rho<\frac{\pi}{2\sqrt \kappa}$ if $\kappa>0$.

Then ${\rm Lip}u$ is in $W^{1,2}_{\rm loc}(\Omega)\cap L_{\rm loc}^\infty(\Omega)$ and  satisfies the following
\begin{equation}\label{equ1.6}
\frac 1 2 \Delta {\rm Lip}^2u  \gs |\nabla {\rm Lip}u|^2- K\cdot {\rm Lip}^2u-  \kappa e_u\cdot  {\rm Lip}^2u
\end{equation}
in the sense of distributions.
\end{thm}

\subsection{The outline of the proof of the Bochner inequality} $\ $

In the following, we would like to give a outline of the proof of Theorem \ref{bochner}. First, by the Chain rule, one easily checks that (\ref{equ1.6}) is equivalent to
\begin{equation}\label{equ1.7}
\Delta {\rm Lip}u  \gs   -K\cdot {\rm Lip}u- \kappa e_u\cdot  {\rm Lip}u
\end{equation}
in the sense of distributions. We will first to show that, for any $q\in(1,2]$
$$\Delta ({\rm Lip}^qu/q)  \gs  - K\cdot {\rm Lip}^qu-  \kappa e_u\cdot  {\rm Lip}^qu\leqno(1.7q)$$
in the sense of distributions, and then we check the limit as $q\to1$ to get (\ref{equ1.7}).

The proof of ($1.7q$) is inspired by the classical Hamilton-Jocabi flow. Recall  that the classical Hamilton-Jacobi equation, given a function $f$:
$$ \frac{\partial v(x,t)}{\partial t} = - |\nabla v(x,t)|^2$$
with $v(x,0) = f(x)$, has a solution (by Hopf-Lax formula)
$$\mathscr H_tf(x):=\inf_{y\in B_{R}}\Big\{\frac{d^2(x,y)}{2t}+f(y)\big)\Big\},\quad \ t>0.$$
The difference of ``time $t$" to the Hamilton-Jacobi flow $\mathscr H_tf(x)$ at $t=0$ gives the gradient $-|\nabla f(x)|^2$. That is, as $t \rightarrow 0^+$,
$$ \frac{\mathscr H_tf(x) - f(x)}{t} \rightarrow - |\nabla f(x)|^2.$$
This suggests to study the Hamilton-Jacobi flow $\mathscr H_tf(x)$ for the gradient estimates of $f$.

In order  to obtain ($1.7q$), we introduce a  family of functions $(f_t)_{t>0}$ by: on a fixed ball $B_R:=B_R(o)$ with $B_{2R}\subset\subset \Omega$, for any $q\in(1,2]$,
\begin{equation}\label{equ1.8}
f_t(x):=\inf_{y\in B_{2R}}\Big\{\frac{d^p(x,y)}{pt^{p-1}}-\phi\big(d_X\big(u(x),u(y)\big)\big) \Big\},\quad \forall\ x\in B_R,\quad \forall\ t>0,
\end{equation}
where $p=q/(q-1)$ and $\phi:[0,1/10]\to \mathbb R$ is a suitable smooth convex function with $\phi(0)=0$ and  $\phi'(0)=1$.

It is easy to check that, for any $x\in B_R$ and any sufficiently small $t$,
 the ``inf"  of \eqref{equ1.8} can be realized by some point $y_{t,x}\in B_{2R}$. The set of all such points is denoted by  $S_t(x)$. Then we put
\begin{equation}\label{equ1.9}
L_t(x):=\min_{y_{t,x}\in S_t(x)}d(x,y_{t,x}) \quad {\rm and}\quad D_t(x):=\frac{L_t^p(x)}{pt^{p-1}}-f_t(x).
\end{equation}
The proof of  ($1.7q$) contains two parts. Without loss of generality, we may assume $\kappa=1$. Firstly,  we shall show  that,  for any given $\varepsilon>0$, $f_t$ satisfies an elliptic inequality
\begin{equation}\label{equ1.10}
\Delta f_t(x) \ls \frac{K}{t^{p-1}}\cdot L^p_t(x)+ (1+\varepsilon)\cdot e_u(x) D_t(x),
\end{equation}
on $B_R$, for any sufficiently small $t>0$,  in the sense of distributions. Secondly, we want to show that, for almost all $x\in B_R$,
\begin{equation}\label{equ1.11}
\lim_{t\to0}\frac{f_t}{t}= -\frac{1}{q}{\rm Lip}^qu,\qquad  \lim_{t\to0^+}\frac{L_t}{t}={\rm Lip}^{q/p}u,\qquad  \lim_{t\to0^+}\frac{D_t}{t}={\rm Lip}^qu.
\end{equation}
The combination of (\ref{equ1.10}) and (\ref{equ1.11}) will yield the inequality ($1.7q$).

In order to prove (\ref{equ1.11}), we recall a generalized  Rademacher  theorem in \cite{kir94}. Let $f:\Omega\to X$ be a Lipschitz map, Kirchheim \cite{kir94} proved  for almost all $x\in \Omega$, that there exists
 a semi-norm,   denoted by $mdf_{x}$ and  called \emph{metric differential}, such that
$$d_X\big(f(\exp_{x}(t\xi)),f(x)\big)-t\cdot mdf_x(\xi)=o(t),$$
for all $\xi\in \mathbb S^{n-1}\subset T_{x}M$. By using this result, one can deduce a representative of point-wise Lipschitz constant of $f$: for almost all $x\in \Omega$,
\begin{equation*}
{\rm Lip}f(x):=\max_{\xi\in\mathbb S^{n-1}}mdf_x(\xi).
\end{equation*}
This suffices to show (\ref{equ1.11}). See Lemma \ref{lem2.9} and Lemma \ref{lem4.4} for the details.

Now, we explain the proof of (\ref{equ1.10}), which is inspired by the recent work \cite{zz17-lip} of the first and the third authors.
For simplicity, we assume $Ric_M\gs0$. We need to show that
$$\Delta  f_t(x) \ls  (1+\varepsilon)e_u(x) D_t(x)+\theta $$
  for sufficiently small $t>0$ and any $\theta>0$  in the sense of distributions. It is a local property, then we need only to consider  the case when $R$ is small.    We argue by contradictions. Suppose that it fails,   by the maximum principle, we have that there exists a domain $U$ and a positive number $\theta_0$  such that $f_t-v$ achieves a strict minimum in $U$, where $v$ is the solution of Dirichlet problem
\begin{equation*}
\Delta v=(1+\varepsilon)e_u(x) D_t(x)+\theta_0 \ \ {\rm in}\ \ U,\quad
v=f_{t} \ \ {\rm on}\ \ \partial U.
\end{equation*}
From the construction of $f_t$, we know that the function
 $$H(x,y):=\frac{d^p(x,y)}{pt^{p-1}}- F(x,y)-v(x).$$
 has a minimum in $U\times B_R$, where $F(x,y):= \phi\big(d_X\big(u(x),u(y)\big)\big)$. We denote one of these minimum points by   $(\bar x,\bar y)$.

Let $T: T_{\bar x }M\to T_{\bar y}M$ be the parallel transportation from $\bar x$ to $\bar y$. We want to  consider the asymptotic behavior of the average
$$\fint_{B_\epsilon(0)}H\big(\exp_{\bar x}(\eta),\exp_{\bar y}(T\eta)\big){\rm d}\eta$$
as $\epsilon\to 0$. According to $Ric_{M}\gs 0$,   by integrating  the second variation of arc-length  over $B_\epsilon(0)$, we have that
\begin{equation}\label{equ1.12}
\fint_{B_\epsilon(0)}\Big(d^p\big(\exp_{\bar x}( \eta), \exp_{\bar y}( T\eta)\big)-d^p(\bar x,\bar y)\Big){\rm d}\eta\ls o(\epsilon^2).
\end{equation}
Notice that $\Delta v=(1+\varepsilon)e_u(x) D_t(x)+\theta_0$ implies that $v$ is smooth near $\bar x$, it follows that
\begin{equation}\label{equ1.13}
-\fint_{B_\epsilon(0)}\Big(v\big(\exp_{\bar x}(\eta)\big)-v(\bar x)\Big){\rm d}\eta\ls -\frac{1}{2(n+2)}\Big[(1+\varepsilon)e_u(\bar x) D_t(\bar x)+\theta_0\Big]\cdot\epsilon^2+o(\epsilon^2).
\end{equation}
Thus, we only need to  show an asymptotic mean value inequality  that
\begin{equation}\label{equ1.14}
 -\fint_{B_\epsilon(0)}\Big(F\big(\exp_{\bar x}( \eta), \exp_{\bar y}( T\eta)\big)-F(\bar x,\bar y)\Big){\rm d}\eta\ls  \frac{1+\varepsilon}{2(n+2)} e_u(\bar x) D_t(\bar x) \cdot\epsilon^2+o(\epsilon^2).
\end{equation}
Indeed, once one has proved (\ref{equ1.14}), the combination of (\ref{equ1.12})-(\ref{equ1.14}) contradicts with the fact that $H(x,y)$ has  a minimum at $(\bar x,\bar y)$, and hence it follows (\ref{equ1.10}).

In order to show (\ref{equ1.14}), we need to choose a suitable function $\phi(s)$ in  (\ref{equ1.8}). In the simplest case that  $\kappa=0$ and $p=q=2$, we can choose directly  $\phi(s)=s$, as  in \cite{zz17-lip}.

 In the case of $\kappa=1$ (and general $q\in(1,2]$),  the definition of $CAT(1)$ suggests us to choose $\phi(s)= 2\sin(s/2)$. However, this is not convex for small $s>0$.
An exact relation in $CAT(1)$-spaces, Lemma \ref{lem2.3}, suggests us to perturb  $2\sin(s/2)$ to
  $$\phi(s)=2\sin(s/2)+4\sin^2(s/2).$$
 Fortunately, this is convex  for small $s>0$.

  Given any $a,b\in\mathbb R$ with $a,b\gs0$, and fixed any $q\in \Omega$, $Q\in X$, we define a function near $q$ by
  \begin{equation*}
 w_{a,b,Q,q}(x):=a\cdot d^2_X\big(u(x),u(q)\big)+b\cdot \cos\big(d_X(u(x),Q)\big).
 \end{equation*}
Since $(X,d_X)$ has curvature $\ls1$, by combining that $e^u_\epsilon$ converge to $e_u$ as $\epsilon\to0$ and the fact
$$\Delta\cos\big(d_X(u(x),Q)\big)\ls -\cos\big(d_X(u(x),Q)\big)\cdot e_u(x),$$
we will be able to  deduce that,  for almost all $q$,   an asymptotic mean value inequality for $w_{a,b,Q,q}$ holds (for some subsequence $\epsilon_j\to0$, see Lemma \ref{lem3.3} for the explicit statements).

 On the other hand, the assumption  $(X,d_X)$ having curvature $\ls1$ implies also that, for any $q_1,q_2$,
 the function   $w_{a_2,b,Q_m,q_1}+ w_{a_1,b,Q_m,q_2}$ touches $-F(\cdot,\cdot)$ by above at $(q_1,q_2)$ for some suitable constants $a_1,a_2,b\gs0$, where $Q_m$ is the mid-point    of $u(q_1)$ and $u(q_2)$  (the details is given in Lemma \ref{lem2.3}). Therefore, we conclude that an asymptotic mean value inequality for $-F(\cdot,\cdot)$ at almost all $(q_1,q_2)$ holds (see Eq.(\ref{equ4.12}) and Lemma \ref{lem3.3} for the explicit formulas).
 First, let us assume briefly that the mentioned asymptotic mean value inequality for $-F(\cdot,\cdot)$ at $(\bar x,\bar y)$. Then we conclude (\ref{equ1.14}) in this case. The primary issue is that there is no reason we can assume that the  asymptotic mean value inequality for $-F(\cdot,\cdot)$ at $(\bar x,\bar y)$. In this case, we will perturb the function $H(x,y)$ to $H_1(x,y):=H(x,y)+\gamma_\delta(x,y)$ by a smooth function $\gamma_\delta(x,y)$, which is arbitrarily small up to two order derivatives, such that the mentioned asymptotic mean value inequality for $-F(\cdot,\cdot)$ holds at one of minimum of  $H_1(x,y)$.   This argument of perturbation can be ensured by  a generalized Jensen's Lemma in the theory of viscosity solutions of second order partial differential equations.

\noindent\textbf{Acknowledgements.} The first and third authors are partially supported by NSFC 11521101, The first author is also partially supported by NSFC 11571374 and by ``National Program for support of Top-notch Young Professionals". The second author is supported by the Academy of Finland. Part of the work was done when the first author visited the Department of Mathematics and Statistics, University of Jyv\"{a}skyl\"{a} for one month in 2016. He would like to thank the department for the hospitality.

\section{Preliminaries}

\subsection{Energy and Sobolev spaces of maps into metric spaces} $\ $

Let $\Omega$ be a bounded open domain of an $n$-dimensional smooth Riemannian manifold $(M,g)$, and let $(X,d_X)$ be a complete metric space. We will  write
$$|xy|:=d(x,y),\quad \forall \ x,y\in M.$$

 Several equivalent notions of Sobolev space  for maps into metric spaces have introduced in \cite{korevaar-schoen93,jost97,heinonen-et01,kuwae-shioya03,ohta04}.
Fix any $p\in[1,\infty)$.
A Borel measurable map $u:\ \Omega\to X$ is said to be in the space $L^p(\Omega,X)$ if it has separable range and, for some (hence, for all) $P\in X$,
$$\int_\Omega d^p_X\big(u(x),P\big)\dv(x)<\infty.$$
We equip with a distance in $L^p(\Omega,X)$ by
$$d^p_{L^p}(u,v):=\int_\Omega d^p_X\big(u(x),v(x)\big)\dv(x), \qquad \forall\ u,v\in L^p(\Omega,X).$$

Denote by $C_0(\Omega)$ the set of continuous functions compactly supported on $\Omega$. Given $ p\in [1,\infty)$ and a map $u\in L^p(\Omega,X),$
 for each $\epsilon>0$, the \emph{approximating energy} $E^u_{p,\epsilon}$ is defined as a functional on $C_0(\Omega)$:
$$E^u_{p,\epsilon}(\phi):=\int_\Omega\phi(x) e^u_{p,\epsilon}(x)\dv(x),\qquad \forall\ \phi\in C_0(\Omega),$$
where the \emph{approximating energy density} is defined by
$$e^u_{p,\epsilon}(x)=e^u_{p,\epsilon,g}(x):=\frac{n+p}{c_{n,p}\cdot\epsilon^n}\int_{B_\epsilon(x)\cap\Omega}\frac{d^p_X\big(u(x),u(y)\big)}{\epsilon^p}{\rm d}v_g(y),$$
and the constant $c_{n,p}=\int_{\mathbb S^{n-1}}|x^1|^p\sigma(dx),$ and $\sigma$ is the
canonical Riemannian volume on $\mathbb S^{n-1}$. In particular, $c_{n,2}=\omega_{n-1}/n$,
 where $\omega_{n-1}$ is the volume of $(n-1)$-sphere $\mathbb S^{n-1}$ with standard metric.
Next, a map $u\in L^p(\Omega,X)$ is said to be in $W^{1,p}(\Omega,X)$ if the energy $E^u_p<\infty$, where
$$E^u_p:=\sup_{\phi\in C_0(\Omega),\ 0\ls \phi\ls1}\Big(\limsup_{\epsilon\to0}E^u_{p,\epsilon}(\phi)\Big).$$

If $1<p<\infty$ and  $u\in W^{1,p}(\Omega,X)$, it was proved in \cite{korevaar-schoen93} that, for each $\phi\in C_0(\Omega)$, the limit
$$E^u_{p}(\phi):=\lim_{\epsilon\to0^+}E^u_{p,\epsilon}(\phi)$$
exists (called \emph{p-th} energy functional of $u$), and that $E^u_p$ is absolutely continuous with respect to the Riemannian volume ${\rm vol}$. Denote the density by $e_{u,p}\in L^1_{\rm loc}(\Omega)$. Moreover, from \cite[Lemma 1.4.2]{korevaar-schoen93}, there exists a constant $C>0$, independent of $\epsilon$ such that
$$E^u_{p,\epsilon}(\phi)\ls  E^u_p\big(C\epsilon\phi+\max_{y\in B_\epsilon(x)}|\phi(y)-\phi(x)|\big)$$
 for any sufficiently small $\epsilon>0$. Thus, by Dunfold-Pettis Theorem, it implies that
 $$e^u_{p,\epsilon}  \to e_{u,p}\ \ {\rm in}\ \  L^1_{\rm loc}(\Omega),\quad {\rm as}\ \epsilon\to0.$$

For the special   case $p=2$, we   write $e_u:=e_{u,2} $ and $E^u:=E^u_2$ for any $u\in W^{1,2}(\Omega,X).$
We summarize some  main properties of $W^{1,2}(\Omega,X)$, which  can be found in   \cite{korevaar-schoen93,kuwae-shioya03}.
\begin{prop}\label{prop2.1}
Let   $u\in W^{1,2}(\Omega, X)$.

\noindent $(1)$ {\rm(Lower semi-continuity)}\indent For any sequence $u_j\to u$ in $L^2(\Omega, X)$ as $j\to\infty$, we have
$$E^u(\phi)\ls \liminf_{j\to\infty} E^{u_j}(\phi),\qquad \forall\ 0\ls \phi\in C_0(\Omega).$$

\noindent $(2)$ {\rm(Equivalence for $X=\mathbb R$)}\indent If $X=\mathbb R$,  the above  space $W^{1,2}(\Omega,\mathbb R)$
is equivalent to the usual Sobolev space $W^{1,2}(\Omega)$.

\noindent $(3)$ {\rm(Weak Poincar\'e inequality, see for example \cite[Theorem 4.2]{kuwae-shioya03})}\indent For any ball $B_R(q)$ with $B_{6R}(q)\subset\subset \Omega$, there exists a constant $C_{n,K,R}>0$ such that the following holds: for any $z\in B_R(q)$ and any $r\in(0,R/2)$, we have
\begin{equation}\label{equ2.1}
\int_{B_r(z)}\int_{B_r(z)}d^2_X\big(u(x),u(y)\big)\dv(x)\dv(y) \ls C_{n,K,R}\cdot r^{n+2}\cdot\int_{B_{6r}(z)}e_u(x)\dv(x).
\end{equation}
\end{prop}
\noindent\emph{Remark 2.1.}\quad By a rescaling argument, one can easily improve the constant $C_{n,K,R}$ in (\ref{equ2.1}) to a constant $C_{n, KR^2}$ depending only on $n$ and $ KR^2$. Indeed, let us consider the rescaling the metric on $M$ by $g_R:=R^{-2}g$. Then we have $Ric_{g_R}\gs -KR^2$ and $dv_{g_R}=R^{-n}dv_g$. By the definition of $e^u_{p,\epsilon,g}$, we get $e^u_{p,R^{-1}\epsilon,g_R}=R^p\cdot e^u_{p,\epsilon,g}$. Therefore, by the definition  of   $e_{u,p}$, the Poincar\'e constant in (\ref{equ2.1}) is invariant with respect to the rescaling $g\mapsto g_R$.


\subsection{$CAT(\kappa)$-spaces} $\  $

Let us review firstly the concept of spaces with curvature bounded above  (globally)  in the sense of Alexandrov.
\begin{defn}[see, for example, \cite{bbi-book,eells-fug01-book}]\label{def2.2}
A geodesic space $(X,d_X)$ is called to be globally \emph{curvature bounded above by $\kappa$} in the sense of Alexandrov, for some $\kappa\in\mathbb R$, denoted by $CAT(\kappa)$, if the
 following comparison property is to hold: Given any triangle $\triangle PQR\subset X$ such that $d_X(P,Q)+d_X(Q,R)+d_X(R,P)<2\pi/\sqrt\kappa$ if $\kappa>0$ and point $S\in QR$ with
 $$d_X(Q,S)=d_X(R,S)= \frac{1}{2} d_X(Q,R),$$
then there exists a comparison triangle $\triangle \bar P\bar Q\bar R$ in the simply connected 2-dimensional space form $\mathbb S^2_\kappa $ with standard metric with sectional curvature $=\kappa$ and point $\bar S\in \bar Q\bar R$ with
 $$d_{\mathbb S^2_\kappa}(\bar Q,\bar S)=d_{\mathbb S^2_\kappa}(\bar R,\bar S)=\frac{1}{2}d_{\mathbb S^2_\kappa}(\bar Q,\bar R)$$
such that
$$d_X(P,S)\ls d_{\mathbb S_\kappa^2}(\bar P,\bar S).$$
\end{defn}
It is obvious that $(X,d_X)$ is a $CAT(\kappa)$-space if and only if the rescaled space $(X,\sqrt\kappa\cdot d_X)$ is a $CAT(1)$-space, for any $\kappa>0$.

We need a lemma, which follows from \cite[Corollary 2.1.3]{korevaar-schoen93}:
\begin{lem}\label{lem2.3}
Let $(X,d_X)$ be an $CAT(1)$ space. Take any ordered sequence $\{P,Q,R,S\}\subset X$, and let point $Q_{m}$ be the mid-point of $QR$.
we denote the distance $d_X(A,B)$ abbreviatedly by $d_{AB}.$
Then, for any $0\ls \alpha\ls 1$ and $\beta>0$, we have
\begin{equation}\label{equ2.2}
\begin{split}
 \frac{1-\alpha}{2}&\Big((2\sin \frac{d_{QR}}{2})^2-(2\sin \frac{d_{PS}}{2})^2\Big)+\alpha(2\sin \frac{d_{QR}}{2})\Big(2\sin \frac{d_{QR}}{2}-2\sin \frac{d_{PS}}{2}\Big) \\
\ls &\Big[1-\frac{1-\alpha}{2}(1-\frac{1}{\beta}) \Big](2\sin\frac{d_{PQ}}{2})^2+2\cos \frac{d_{QR}}{2}\Big(\cos d_{PQ_m}-\cos d_{QQ_m}\Big)\\
 &\  +\Big[1-\frac{1-\alpha}{2}(1- \beta ) \Big](2\sin\frac{d_{RS}}{2})^2+2\cos \frac{d_{QR}}{2}\Big(\cos d_{SQ_m}-\cos d_{RQ_m}\Big).
 \end{split}
 \end{equation}
\end{lem}
\begin{proof}
Consider the embedding $X$ into the cone $C(X)$ with the cone metric $|\cdot\ \cdot\ |_C$. Then $C(X)$ has non-positive curvature in the sense of Alexandrov. Denote that
$$ \bar P=(P,1),\quad \bar Q=(Q,1),\quad  \bar S=(S,1),\quad \bar R=(R,1)  \quad {\rm and}\quad  \bar Q_m=(Q_m,1). $$
It is clear that the midpoint of $\bar Q,\bar R$ in $C(X)$ is
$$ \bar T=(Q_m,\cos\frac{d_{QR}}{2}).$$
From the equation (2.1v) in \cite[Corollary 2.1.3]{korevaar-schoen93} (by taking $t=1/2$ there), we get, for each $\alpha\in[0,1]$, that
\begin{equation*}
\begin{split}
|\bar T\bar P|^2_C+|\bar T\bar S|^2_C\ls & |\bar P\bar Q|^2_C+|\bar R\bar S|^2_C+\frac 1 2\big(|\bar S\bar P|_C^2-|\bar Q\bar R|_C^2\big)+\frac 1 2|\bar Q\bar R|^2_C\\
& -\frac 1 2\Big(\alpha\big(|\bar S\bar P|_C-|\bar Q\bar R|_C\big)^2+(1-\alpha)\big(|\bar R\bar S|_C-|\bar P\bar Q|_C\big)^2\Big).
\end{split}
\end{equation*}
Notice that $|\bar Q\bar R|_C=2|\bar T\bar Q|_C=2|\bar T\bar R|_C$ and that
$$\big(|\bar S\bar P|^2_C-|\bar Q\bar R|_C^2\big)-\alpha \big(|\bar S\bar P|_C-|\bar Q\bar R|_C\big)^2=(1-\alpha)\big(|\bar S\bar P|^2_C-|\bar Q\bar R|_C^2\big)+2\alpha|\bar Q\bar R|_C\big(|\bar S\bar P|_C-|\bar Q\bar R|_C\big) .$$
Therefore, we obtain
\begin{equation}\label{equ2.3}
\begin{split}
\frac{1-\alpha}{2}\big(|\bar Q\bar R|_C^2&-|\bar S\bar P|^2_C\big)+\alpha|\bar Q\bar R|_C\big(|\bar Q\bar R|_C-|\bar S\bar P|_C\big)\\
\ls& |\bar P\bar Q|^2_C+|\bar T\bar Q|^2_C-|\bar T\bar P|^2_C+ |\bar S\bar R|^2_C+|\bar T\bar R|^2_C-|\bar T\bar S|^2_C\\
&\quad-\frac{1-\alpha}{2}\big(|\bar R\bar S|_C-|\bar P\bar Q|_C\big)^2\\
\ls& |\bar P\bar Q|^2_C+|\bar T\bar Q|^2_C-|\bar T\bar P|^2_C+ |\bar S\bar R|^2_C+|\bar T\bar R|^2_C-|\bar T\bar S|^2_C\\
&\quad  -\frac{1-\alpha}{2}\big(|\bar R\bar S|^2_C+|\bar P\bar Q|^2_C-\beta |\bar R\bar S|^2_C-\frac 1 \beta |\bar P\bar Q|^2_C\big)\\
=&  |\bar P\bar Q|^2_C\Big(1-\frac{1-\alpha}{2}\big(1-\frac 1 \beta\big)\Big)+|\bar T\bar Q|^2_C-|\bar T\bar P|^2_C\\
&\quad + |\bar S\bar R|^2_C\Big(1-\frac{1-\alpha}{2}\big(1- \beta\big)\Big)+|\bar T\bar R|^2_C-|\bar T\bar S|^2_C
\end{split}
\end{equation}
for any $\beta>0$, where we have used $2  |\bar R\bar S|_C\cdot |\bar P\bar Q|_C\ls \beta |\bar R\bar S|^2_C+\frac 1 \beta |\bar P\bar Q|^2_C$. By recalling
 the definition of the cone metric $|\cdot\cdot\ |_C$, we have
\begin{equation*}
\begin{split}
&|\bar Q\bar R|_C=2\sin\frac{d_{QR}}{2}, \qquad |\bar S\bar P|_C=2\sin\frac{d_{SP}}{2}, \\
&|\bar P\bar Q|_C=2\sin\frac{d_{PQ}}{2}, \qquad |\bar R\bar S|_C=2\sin\frac{d_{RS}}{2},\\
&  |\bar T\bar Q|_C=|\bar T\bar R|_C=\frac{|\bar Q\bar R|_C}{2}=\sin\frac{d_{QR}}{2}
\end{split}
\end{equation*}
and (by noticing that $|O\bar T|_C=\cos\frac{d_{QR}}{2}$,)
\begin{equation*}
\begin{split}
&|\bar T\bar P|_C^2=1+\cos^2\frac{d_{QR}}{2}-2\cos\frac{d_{QR}}{2}\cos  d_{PQ_m}, \\
&|\bar T\bar S|_C^2=1+\cos^2\frac{d_{QR}}{2}-2\cos\frac{d_{QR}}{2}\cos  d_{SQ_m}.
\end{split}
\end{equation*}
Then
\begin{equation*}
\begin{split} |\bar T\bar Q|^2_C-|\bar T\bar P|^2_C&=\sin^2\frac{d_{QR}}{2}-1-\cos^2\frac{d_{QR}}{2}+2\cos\frac{d_{QR}}{2}\cos  d_{PQ_m}\\
&=2\cos\frac{d_{QR}}{2}\Big(\cos d_{PQ_m}- \cos \frac{d_{QR}}{2}\Big)\\
&=2\cos\frac{d_{QR}}{2}\Big(\cos d_{PQ_m}- \cos d_{QQ_m}\Big),
\end{split}
\end{equation*}
where we have used $d_{QQ_m}=\frac{d_{QR}}{2}.$ Similarly, we have
\begin{equation*}
 |\bar T\bar R|^2_C-|\bar T\bar S|^2_C =2\cos\frac{d_{QR}}{2}\Big(\cos d_{SQ_m}- \cos d_{RQ_m}\Big).\qquad
\end{equation*}
Therefore, the combination of these and the estimate (\ref{equ2.3}) implies the desired (\ref{equ2.2}). The proof is completed.
\end{proof}

\subsection{Harmonic maps} $\  $

In the following, we always assume that $\Omega$ is a bounded domain in an $n$-dimensional smooth Riemannian manifold $(M,g)$ with $Ric_M\gs -K$ for some $K\gs0$  and that $(X,d_X)$ is a  $CAT(\kappa)$ space for some $\kappa\gs0$.

Given any $\phi\in W^{1,2}(\Omega,X)$, we set
$$W^{1,2}_\phi(\Omega,X):=\big\{u\in W^{1,2}(\Omega,X):\ d_X\big(u(x),\phi(x)\big)\in W^{1,2}_0(\Omega)\big\}.$$
Using the variation method, it was proved in \cite{jost97,lin97} that there exists a unique $u\in W^{1,2}_\phi(\Omega,X)$ which is a minimizer of energy
$E_2^u$ in $W^{1,2}_\phi(\Omega,X)$. That is, the energy $E^u:=E_2^u=E^u_2(\Omega)$ of $u$ satisfies
$$E^u=\inf_w\big\{E^w:\ w\in W^{1,2}_\phi(\Omega,X)\big\}.$$
Such an energy minimizing map is called a \emph{harmonic map}.

 The basic existence and regularity were given by Korevaar-Schoen in \cite{korevaar-schoen93} for $\kappa\ls 0$ and by Serbinowski in \cite{serbinowski95} for $\kappa>0$. We state their regularity result in the case $\kappa>0$ (see also \cite[Theorem 2.3]{mese02}):
\begin{thm}[\cite{korevaar-schoen93,serbinowski95}]\label{thm2.4}
Let $u$ be a harmonic map  from $\Omega$ to $X$. Assume that its image $u(\Omega)$ is contained in a ball $B_\rho \subset X$ with radius $\rho<\frac{\pi}{2\sqrt\kappa}$ if $\kappa>0$. Then $u$ is locally Lipschitz continuous in the interior of $\Omega$. (Note that  the local Lipschitz constant of $u$ near a point  $x\in \Omega$ depends on the   $C^1$-norm of metric $g$ near $x$.)
\end{thm}

We need also the following property:
\begin{lem} {\rm (Serbinowski \cite[Proposition 1.17]{serbinowski95}, Fuglede \cite[Lemma 2]{fug08}).}\label{lem2.5}
Let $\kappa>0$.  Assume that its image $u(\Omega)$ is contained in a ball $B_\rho(P)\subset X$ with radius $\rho<\frac{\pi}{2\sqrt\kappa}$. Then the function
$f_P(x):= \cos\big(\sqrt \kappa \cdot d_X\big(u(x),P\big)\big)$
 satisfies $  f_P\in W^{1,2}(\Omega)$ and
\begin{equation}\label{equ2.4}
 \Delta {f_P}\ls - \kappa\cdot f_P\cdot e_u
\end{equation}
in the sense of distributions. If $\kappa=0$, then for any $P\in X$ we have $\Delta d^2_X\big(P,u(x)\big)\gs 2e_u$ in the sense of distributions.
\end{lem}
Recall that
$$  {\rm Lip}u(x) =\limsup_{y\to x}\frac{d_X\big(u(x),u(y)\big)}{|xy|}=\limsup_{r\to 0}\sup_{y\in B_x(r)}\frac{d_X\big(u(x),u(y)\big)}{r}.$$
The above lemma implies the following  point-wise estimates, which is a corollary of the  mean value inequality for subharmonic functions.
\begin{cor}\label{cor2.6}
 Let $u$ be a harmonic map  from $\Omega$ to $X$. Assume that its image $u(\Omega)$ is contained in a ball $B_\rho \subset X$ with radius $\rho<\frac{\pi}{2\sqrt\kappa}$ if $\kappa>0$. Then there exists a constant $C=C(n,\sqrt KR)$ depending only on $n$ and $\sqrt KR$  such that:   for any ball $B_{R}$ with $B_{2R}\subset\subset \Omega$, we have
 \begin{equation*}
{\rm Lip}^2u(x)\ls C\cdot e_u(x),\qquad {\rm for\ almost\ all}\ \ x\in  B_{R/6}.
\end{equation*}
\end{cor}
\begin{proof}
For the case $\kappa=0$, this is Theorem 5.5 in \cite{zz17-lip}. We need only to show the assertion for the case $\kappa>0$. Without loss of the generality, we can assume $\kappa=1$ in this case.  The argument is similar to  the proof of Theorem 5.5 in \cite{zz17-lip}.

$(i).$ Fix any $z$ with $B_{2R}(z)\subset\subset \Omega$.   From the continuity of $u$, there exists a small neighborhood $O$ of $z$ such that ${\rm diam}u(O) <\pi/2$ and $O\subset B_R(z)$, where ${\rm diam}u(O)$ is the diameter of $u(O)$.

By using Lemma \ref{lem2.5} and the fact that $|\nabla d_X(u(x),P)| \ls e_u$ for any fixed $P\in X$, it is easy the check that $\Delta d_X(u(x),u(y_0))\gs0$ on $O$  for any fixed $y_0\in  O$, in the sense of distributions.
Let $\Delta^{(2)}$ be  the Laplace-Beltrami operator on $M\times M$, the product manifold (with the product metric and  the product measure). Consider the function $\rho_u:=d_X\big(u(x),u(y)\big)$ on $O\times O$.
Hence, we obtain
 \begin{equation}\label{equ2.5}
 \Delta^{(2)}\rho_u(x,y)\gs 0 \quad {\rm on}\quad O\times O
 \end{equation}
 in the sense of distributions  (see the step (iii) in the proof of \cite[Proposition 5.4]{zz17-lip} for the details).

  From the mean value inequality for subharmonic functions  on  $O\times O$ (see \cite[Theorem 6.2 of Chapter II]{schoen-yau94}), we  conclude that, for any ball $B_r((z_1,z_2))$  with $B_{2r}((z_1,z_2))\subset\subset O\times O$,
\begin{equation}\label{equ2.6}
\begin{split}
\sup_{(x,y)\in B_{r}((z_1,z_2))}\rho_u^2(x,y)&\ls C_12^{C_2(1+\sqrt Kr)}\cdot \fint_{B_{2r}((z_1,z_2))} \rho_u^2(x,y)dv_g(x)dv_g(y)\\
&\ls C_3(n,\sqrt KR)\cdot \fint_{B_{2r}((z_1,z_2))} \rho_u^2(x,y)dv_g(x)dv_g(y),
\end{split}
\end{equation}
where the constants $C_1,C_2$ depend only on $n$, and $C_3(n,\sqrt K R)=C_12^{C_2(1+\sqrt KR)}$.

$(ii).$
Since $B_{2r}((z,z))\subset B_{2r}(z)\times B_{2r}(z)$, the Poincar\'e inequality for $W^{1,2}(\Omega, X)$-maps (see \cite{kuwae-shioya03}, and also Proposition 2.1 (3) and Remark 2.1) states that
the RHS of (\ref{equ2.6})  for $z_1=z_2=z$  can be dominated by $C_4(n,\sqrt KR)\cdot r^{n+2}\int_{B_{12r}(z)}e_u(x)\dv(x).$ Therefore, we have
\begin{equation}\label{equ2.7}
\begin{split}
\sup_{y\in B_r(z)}\frac{\rho_u^2(z,y)}{r^2}&\ls\sup_{(x,y)\in B_{r}((z,z))}\frac{\rho_u^2(x,y)}{r^2}\\
&\ls C_3C_4\cdot\frac{r^n\cdot {\rm vol}(B_z(12r))}{vol\big(B_{2r}((z,z))\subset \Omega\times\Omega\big)} \fint_{B_{12r}(z)} e_u(x)dv_g(x)\\
&\ls C_5(n,\sqrt KR)\cdot  \fint_{B_{12r}(z)} e_u(x)dv_g(x),
\end{split}
\end{equation}
where we have used Bishop-Gromov inequality and  $vol\big(B_{2r}((z,z))\big)\gs vol^2\big(B_r(z)\big)$.

Notice that $\lim_{r\to0}\fint_{B_{12r}(z)} e_u(x)dv_g(x)=e_u(z)$ for almost all $z\in B_{R/6}$ and that
$ {\rm Lip}u(z)=\limsup_{r\to0}\sup_{y\in B_{r}(z)}\rho_u(y,z)/r$.
 By letting   $r\to0$ in (\ref{equ2.7}), it follows the desired estimate.
\end{proof}

\subsection{Generalized Rademacher theorem for Lipschitz maps}$\ $

Let $\Omega$ be a bounded domain of an $n$-dimensional Riemannian manifold $(M,g)$. Recall that the classical Rademacher  theorem states that any Lipschitz function $f:\Omega \to \mathbb R$ is differentiable at almost all $x\in \Omega$.

For our purpose, we have to consider the differentiability of  maps into a metric space $(X,d_X)$.
Let us recall the notion of \emph{metric differential} for maps from $\Omega$ into a metric space, which was introduced by Kirchheim in \cite{kir94}.
\begin{defn}\label{def2.8}
We say that a map $f:\Omega\to X$ is \emph{metrically differentiable} at $x_0$ if there exists
 a semi-norm $\|\cdot\|_{x_0}$ in $T_{x_0}M:=\R^n$ such that
$$d_X\big(f(\exp_{x_0}(t\xi)),f(x_0)\big)-t\cdot\|\xi\|_{x_0}=o(t),$$
for all $\xi\in \mathbb S^{n-1}\subset T_{x_0}M$.
 This semi-norm will be called the \emph{metric differential} and be denoted by $mdf_{x_0}$.
\end{defn}
The following generalized Rademacher's theorem for maps was given in \cite{kir94}.
\begin{thm}[Kirchheim \cite{kir94}]\label{thm2.8}
Any Lipschitz map $f:\Omega\to X$ is metrically differentiable at almost all $x\in \Omega$.
\end{thm}

If a Lipschitz continuous map $f:\Omega\to X$ is metrically differentiable at $x$,  we put
\begin{equation}
G_f(x):=\max_{\xi\in\mathbb S^{n-1}}mdf_x(\xi).
\end{equation}
\begin{lem}\label{lem2.9}
 Let $f:\Omega\to X$ be a Lipschitz function. If $f$ is metrically differentiable at $x$, then we have
\begin{equation}\label{equ2.8}
G_f(x)={\rm Lip}f(x).
\end{equation}
\end{lem}
\begin{proof}
From the definition of $G_f(x)$, it is clear that $G_f(x)\ls {\rm Lip}f(x).$

For the converse, we choose a sequence of points $\{y_j:=\exp_x(t_j\xi_j)\}_{j=1}^\infty\subset \Omega$ such that $\lim_{j\to\infty}t_j=0$, $|\xi_j|=1$, and
$${\rm Lip}f(x)=\lim_{j\to \infty}\frac{d_X\big(f(y_j),f(x)\big)}{t_j}.$$
Since $f$ is metrically differentiable at $x$, we have
\begin{equation*}
\begin{split}
d_X\big(f(y_j),f(x)\big)&=mdf_x(\xi_j)\cdot t_j +o(t_j),
\end{split}
\end{equation*}
From  the definition of $G_f(x)$, we have
$${\rm Lip}f(x)=\lim_{j\to \infty}mdf_x(\xi_j)\ls G_f(x).$$
 The proof is complete.
\end{proof}

\section{An asymptotic mean value inequality}

We will consider some asymptotic behaviors of harmonic maps from a domain of smooth Riemannian manifold  to a $CAT(1)$-space. Let us begin with the following mean value property, which is similar to  Proposition 2.1 of Chapter I in \cite{schoen-yau94}.
\begin{lem}\label{lem3.1}
Let $(M,g)$ be an $n$-dimensional   Riemannian manifold with $Ric_M\gs -K$ for some $K\in\mathbb R$. Suppose that $f$ is a Lipschitz  function on an open subset $\Omega\subset M$, $f\gs0$, and $\Delta f\ls g\in L_{\rm loc}^1(\Omega)$ in the sense of distributions. Then for any $p\in \Omega$ and $R>0$ with $B_R(p)\subset\subset \Omega$,
\begin{equation}\label{equ3.1}
\frac{1}{A_K(R)}\int_{\partial B_R(p)}f\ls f(p)+ \int_0^R\frac{\int_{B_r(p)}g(x) \dv(x)}{A_K(r)}dr,
\end{equation}
where $A_K(r)$  is the area of a geodesic sphere of radius $r$ in the simply connected space form of constant curvature $-K/(n-1).$
\end{lem}
\begin{proof}
Since $\Delta f\ls g$, we have by divergence theorem that
\begin{equation}\label{equ3.2}
\begin{split}
\int_{B_r(p)}g(x)\dv(x)&\gs \int_{B_r(p)}\Delta f\dv=\int_{\partial B_r(p)}\frac{\partial f}{\partial r}\\
& =\frac{\partial}{\partial r}\int_{\partial B_r(p)}f-\int_{\partial B_r(p)}Hf,
\end{split}
\end{equation}
where $0<r<R$, and $H$ is the mean curvature of $\partial B_r(p)$ with resect to $\partial/\partial r$. The standard comparison theorem  asserts that
$$H(x)\ls (n-1)\cot_K(r) =\frac{A_K'(r)}{A_K(r)},\qquad \forall\ x\in \partial B_r(p).$$
Therefore, it follows from (\ref{equ3.2}) and the assumption $f\gs0$ that
$$\frac{\int_{B_r(p)}g(x) \dv(x)}{A_K(r)}\gs  \frac{\partial}{\partial r}\frac{\int_{\partial B_r(p)}f}{A_K(r)}.$$
Notice that $\lim_{r\to0}\frac{\int_{\partial B_r(p)}f}{A_K(r)}=f(p)$. Integrating both sides of the above inequality with respect to $r$ over $(0,R)$, we conclude that (\ref{equ3.1}) holds.
\end{proof}
Now we consider the case that $f$ needs not to be nonnegative.
\begin{cor}\label{cor3.2}
Let $(M,g)$ be an $n$-dimensional   Riemannian manifold with $Ric_M\gs- K$ for some $K\in\mathbb R$. Suppose that $f$ is a Lipschitz  function on an open subset $\Omega\subset M$, and $\Delta f\ls g\in L^1_{\rm loc}(\Omega)$ in the sense of distributions. Then for any $p\in \Omega$ and   $R>0$ with $B_R(p)\subset\subset\Omega$,
\begin{equation}\label{equ3.3}
\begin{split}
\frac{1}{V_K(R)}\int_{B_R(p)}\Big(f(x)-f(p)\Big)\dv(x)\ls & R\cdot Lip_{B_R(p)}f\cdot\Big(1-\frac{{\rm vol}\big(B_R(p)\big)}{V_K(R)}\Big)\\
& +\frac{1}{V_K(R)}\int_0^R A_K(r)\int_0^r\frac{\int_{B_s(p)}g(x) \dv(x)}{A_K(s)}{\rm d}s{\rm d}r.
\end{split}
\end{equation}
where $V_K(r)$ is the is the volume of a geodesic ball of radius $r$ in the space form of constant curvature $-K/(n-1),$ and
$$Lip_{B_R(p)}f:=\sup_{x,y\in B_R(p)}\frac{|f(x)-f(y)|}{|xy|}.$$

In particular, if $p$ is a Lebesgue point of $g$, then the following asymptotic mean value inequality holds
\begin{equation}\label{equ3.4}
\frac{1}{V_K(R)}\int_{B_R(p)}\Big(f(x)-f(p)\Big)\dv(x)\ls \frac{g(p)}{2(n+2)}\cdot R^2+o(R^2)\quad {\rm as}\ \ R\to0.
\end{equation}
\end{cor}
\begin{proof}
We consider the function $h(x):=f(x)-f(p)$. By applying   Lemma  \ref{lem3.1} to nonnegative function
$$h_r(x):=h(x)-\inf_{y\in B_r(p)}h(y)$$
on $B_r(p)$, $0<r<R$, we have
$$\frac{1}{A_K(r)}\int_{\partial B_r(p)}h_r\ls h_r(p)+ G_K(r)= -\inf_{y\in B_r(p)}h(y)+G_K(r),$$
where
$$ G_K(r):=\int_0^r\frac{\int_{B_s(p)}g(x) \dv(x)}{A_K(s)}ds.$$
Denote by $A(r):={\rm vol}_{n-1}(\partial B_r(P)\subset M)$. We get
\begin{equation*}
\int_{\partial B_r(p)}h \ls   -\inf_{y\in B_r(p)}h(y)\cdot \Big(A_K(r)-A(r)\Big)+G_K(r)\cdot A_K(r).
\end{equation*}
Remark that $h(p)=0$ and $Lip_{B_r(p)}h\ls Lip_{B_R(p)}h=Lip_{B_R(p)}f$, so we have
$$-\inf_{y\in B_r(p)}h(y)\ls-\inf_{y\in B_R(p)}h(y)\ls R\cdot Lip_{B_R(p)}f.$$
 The above two inequalities implies
\begin{equation*}
\int_{\partial B_r(p)}h \ls   R\cdot Lip_{B_R(p)}f\cdot \Big(A_K(r)-A(r)\Big)+G_K(r)\cdot A_K(r).
\end{equation*}
where we have used the Bishop inequality $A(r)\ls A_K(r)$ for all $r\in(0,R).$  Integrating both sides of the above inequality with respect to $r$ over $(0,R)$,  and then dividing by $V_K(R)$, we get (\ref{equ3.3}).

Suppose  that   $p$ is a Lebesgue point of $g$, i.e.,
\begin{equation}
\lim_{R\to0}\fint_{B_R(p)}g(x)\dv(x)=g(p).
\end{equation}
Notice that
\begin{equation}
\frac{{\rm vol}\big(B_R(p)\big)}{V_K(R)}=1+O(R^2)\quad {\rm as}\ \ R\to0.
\end{equation}
It follows that
$$ R\cdot Lip_{B_R(p)}f\cdot\Big(1-\frac{{\rm vol}\big(B_R(p)\big)}{V_K(R)}\Big)=O(R^3)\quad {\rm as}\ \ R\to0$$
and that
$$ \frac{1}{V_K(R)}\int_{B_R(p)}g(x)\dv(x)=g(p)+o(1)\quad {\rm as}\ \ R\to0.$$
Thus, by a direct calculation (noticing that $A_K(t)=\omega_{n-1}\cdot t^{n-1}+O(t^n)$ as $t\to0$), we get
\begin{equation*}
\begin{split}
\frac{1}{V_K(R)}\int_0^R &A_K(r) \int_0^r\frac{\int_{B_s(p)}g(x) \dv(x)}{A_K(s)}{\rm d}s{\rm d}r\\
=&\frac{1}{V_K(R)}\int_0^R A_K(r)\int_0^r\frac{V_K(s)\cdot\big(g(p)+o(1)\big)}{A_K(s)}{\rm d}s{\rm d}r\\
=&\frac{1}{V_K(R)}\int_0^R A_K(r)\int_0^r\frac{s}{n}\big(1 +o(1)\big)\big(g(p) +o(1)\big){\rm d}s{\rm d}r\\
=&\frac{g(p)}{2(n+2)}\cdot R^2+o(R^2)
\end{split}
\end{equation*}
as $R\to0$. The proof is finished.
\end{proof}

At last in this section, we want to use the above asymptotic mean value inequality to harmonic maps to a metric space with curvature bounded above.
\begin{lem}\label{lem3.3}
 Let $\Omega$ be a bounded domain in an $n$-dimensional smooth Riemannian manifold $(M,g)$ and that $(X,d_X)$ is a  $CAT(1)$-space. Suppose that $u$ is a harmonic map  from $\Omega$ to $X$.  Given any $a,b\in\mathbb R$ with $a,b\gs0$, and any $q\in \Omega$, $Q\in X$, we put
  \begin{equation}\label{equ3.7}
 w_{a,b,Q,q}(x):=a\cdot d^2_X\big(u(x),u(q)\big)+b\cdot \cos\big(d_X(u(x),Q)\big).
 \end{equation}

  Then there exists a  sequence   $\{\epsilon_j\}_{j\in\mathbb N}$ with $\epsilon_j\to 0$ as $j\to\infty$ and a subset $\mathscr N\subset \Omega$ with zero measure such that
\begin{equation}\label{equ3.8}
\begin{split}
\int_{B_{\epsilon_j}(0)}&\Big[w_{a,b,P,x_0}\big(\exp_{x_0}(\eta)\big)- w_{a,b,P,x_0}(x_0) \Big]{\rm d}\eta\\
&\ls  \Big(2a-b\cdot\cos\big(  d_X (u(x_0),P )\big)\Big)   \frac{  \omega_{n-1}}{2n(n+2)}\cdot e_u(x_0)\cdot \epsilon_j^{n+2}+o( \epsilon_j^{n+2})
\end{split}
\end{equation}
 for any $x_0\in \Omega\backslash\mathscr N$  and for any  $P\in X$ such that   the image $u(\Omega)$ is contained in a ball $B_\rho(P)\subset X$ with radius $\rho<\frac{\pi}{2}$, and for every  $a,b \gs0$.
\end{lem}
\begin{proof}
Recall  that $e^u_{2,\epsilon}\to e_u$ in $L^1_{\rm loc}(\Omega)$ as $\epsilon\to0$. Therefore,    there exists a sequence $\{\epsilon_j\}_{j\in\mathbb N}$ with $\epsilon_j\to 0$ as $j\to\infty$ such that
$$\lim_{j\to\infty}e^u_{2,\epsilon_j}(x_0)=e_u(x_0)\quad {\rm for\ almost\ all}\ x_0\in\Omega.$$
By the definition of the approximating energy density, it follows   that
$$\int_{B_{\epsilon_j}(x_0)}d^2_X\big(u(x),u(x_0)\big)\dv(x)=\frac{\omega_{n-1}}{n(n+2)}\cdot e_u(x_0)\cdot \epsilon_j^{n+2}+o(\epsilon_j^{n+2})$$
for almost all point $x_0\in \Omega$.
On the other hand, we have by Lemma \ref{lem2.5} and Corollary \ref{cor3.2}  that
 \begin{equation*}
\begin{split}
\int_{B_{\epsilon_j}(x_0)} \Big[&\cos\big(  d_X\big(u(x),P\big)\big)-\cos\big(d_X\big(u(x_0),P\big)\big)\Big]\dv(x)\\
 & \ls V_K(\epsilon_j)\cdot \Big[-\frac{\cos\big(  d_X\big(u(x_0),P\big)\big) }{2(n+2)}\cdot e_u(x_0)\cdot \epsilon_j^{2}+o( \epsilon_j^{2})\Big]\\
& = -\frac{\cos\big(  d_X\big(u(x_0),P\big)\big)\cdot \omega_{n-1}}{2n(n+2)}\cdot e_u(x_0)\cdot \epsilon_j^{n+2}+o( \epsilon_j^{n+2})
\end{split}
\end{equation*}
for all Lebesgue points $x_0 $ of $e_u$, and for all $P\in X$ such that the image $u(\Omega)$ is contained in a ball $B_\rho(P)\subset X$ with radius $\rho<\frac{\pi}{2}$. Here we have used that $V_K(\epsilon_j)=\frac{\omega_{n-1}}{n}\cdot \epsilon_j^n+o(\epsilon_j^n).$ Thus, for almost all $x_0\in \Omega$, we have
 \begin{equation}\label{equ3.9}
\begin{split}
\int_{B_{\epsilon_j}(x_0)}&\Big[w_{a,b,P,x_0} (x)- w_{a,b,P,x_0}(x_0) \Big]\dv(x)\\
 &\ls \Big(2a-b\cdot\cos\big(  d_X (u(x_0),P )\big)\Big)   \frac{  \omega_{n-1}}{2n(n+2)}\cdot e_u(x_0)\cdot \epsilon_j^{n+2}+o( \epsilon_j^{n+2})
\end{split}
\end{equation}
and any  $P\in X$, and for every  $a,b\in \mathbb R$ with $b\gs0$.

At last, we consider the exponential map $\exp_{x_0}: B_{\epsilon_j}(0)\subset \mathbb R^n\to B_{\epsilon_j}(x_0)$. It is well known
$$\frac{\dv}{{\rm d}\eta}=1+o(\epsilon_j).$$
Thus, for any general Lipschitz function $h$, we have, as $r\to0$,
\begin{equation*}
\begin{split}
\int_{B_r(0)}&\Big(h(\exp_{x_0}(\eta))-h(x_0)\Big){\rm d}\eta\\
&\ls \int_{B_r(x_0)}\Big(h(x)-h(x_0)\Big)(1+o(r))\dv(x)\\
&\ls \int_{B_r(x_0)}\Big(h(x)-h(x_0)\Big)\dv(x)+o(r)\cdot \int_{B_r(x_0)}\Big|h(x)-h(x_0)\Big| \dv(x)\\
&\ls \int_{B_r(x_0)}\Big(h(x)-h(x_0)\Big)\dv(x)+o(r)\cdot O(r)\cdot O(r^{n})
\end{split}
\end{equation*}
 By using this to  $w_{a,b,P,x_0}$ and combining with (\ref{equ3.9}),  we obtain Eq. (\ref{equ3.8}).
\end{proof}

\section{The Bochner inequality for harmonic maps into $CAT(\kappa)$-spaces}

 Let $\Omega$ be a bounded domain in an $n$-dimensional smooth Riemannian $(M,g)$ with $Ric_{M}\gs- K$ for some   $K\gs0$,
and let $(X,d_X)$ be a complete $CAT(\kappa)$-space for some $\kappa>0$.

\emph{In this section,
we always assume that $u:\Omega\to X$ is a harmonic map with the image $Im(u)$ containing in a ball $B_\rho(Q_0)\subset Y$ with $\rho<\pi/(2\sqrt\kappa)$.
From Theorem \ref{thm2.4}, we know
that $u$ is local Lipschitz continuous on $\Omega$.}

\subsection{Auxiliary functions}$\  $

In this subsection, we will introduce a family of auxiliary functions.

Fix  $p\in(1,\infty)$ and a ball $B_R(o)$ such that $B_{2R}(o)\subset\subset \Omega$. Denote by $B_R:=B_R(o)$ and by
$$  \ell_0:=Lip_{B_{2R}}u=\sup_{x,y\in B_{2R} ,\ x\not=y}\frac{d_X(u(x),u(y))}{|xy|}    <\infty.$$
We  introduce a family of auxiliary functions   $f_t(x)$ on $B_R$ as follows: for any $t>0$, we define
\begin{equation}\label{equ4.1}
f_t(x):=\inf_{y\in B_{2R}}\Big\{\frac{|xy|^p}{pt^{p-1}}-F(x,y) \Big\},\qquad x\in B_R,
\end{equation}
where
$$F(x,y):=2\sin\frac{d_X\big(u(x),u(y)\big)}{2}+4 \sin^2\frac{d_X(u(x),u(y))}{2}.$$
It is clear that  $F(x,y)\ls 6$ and that  (by taking $y=x$)
\begin{equation}\label{equ4.2}
0\gs f_t(x)\gs-6.
\end{equation}

For any  $0<t<t_*(:= (R^p/6p)^{1/(p-1)})$, it is clear that the
 the ``inf"  of \eqref{equ4.1} can be achieved, i.e., for any $x\in B_R$,
$$S_t(x):=\Big\{y\in B_{2R}\ |\ f_t(x)=\frac{|xy|^p}{pt^{p-1}}-F(x,y) \Big\}\not= \varnothing.$$
Since $F(x,\cdot)$ is continuous on $B_{2R}$, it follows that $S_t(x)$ is close.
Fix any small $t\in (0,t_*)$. We define two functions on $B_R$
\begin{equation}\label{equ4.3}
L_t(x):=\min_{y\in S_t(x)}|xy|\qquad {\rm and}\qquad D_t(x):=\frac{L_t^p(x)}{pt^{p-1}}-f_t(x).
\end{equation}
We give some basic properties of these functions.
\begin{lem}\label{lem4.1}
 For any $t\in (0,t_*)$, we have the following properties.\\
(1)\quad $f_t$ is Lipschitz continuous on $B_R$;\\
(2)\quad Both $L_t$ and $D_t$ are lower semi-continuous on $B_R$;\\
(3)\quad There exists a constant $C=C( p, \ell_0,\kappa)>0$ such that, for any  $t\in (0,t_*)$,
$$L_t\ls Ct,\quad   D_t\ls C t\quad{\rm and}\quad-f_t\ls C  t\quad{\rm on}\ \ B_R.$$
\end{lem}

\begin{proof}
(1) Take any  $x_1,x_2\in B_R$. From the definition of $f_t$, we by choosing some $y_2\in S_t(x_2)$  have that
$$f_t(x_1)-f_t(x_2)\ls \frac{|x_1y_2|^p-|x_2y_2|^p}{pt^{p-1}}-\big(F(x_1,y_2)-F(x_2,y_2)\big).$$
Noticing that both $|\cdot\ \cdot\ |^p$ and $F(\cdot,\cdot)$ are Lipschitz, we conclude that there exists some constant $C_t>0$ such that $f_t(x_1)-f_t(x_2)\ls C_t|x_1x_2|.$ That is,
$f_t$ is Lipschitz continuous.

(2) From the definition of $L_t$, we know that $L_t$ is lower semi-continuous. By using that $f_t$ is continuous, we get that $D_t$ is also lower semi-continuous.

(3) We take some $y_t\in S_t(x)$ such that $L_t(x)=|xy_t|.$ We have
\begin{equation}\label{equ4.4}
\begin{split}
D_t(x)&=\frac{L^p_t(x)}{pt^{p-1}}-f_t(x) =F(x,y_t)\ls   d_X(u(x),u(y_t))+  d^2_X(u(x),u(y_t))\\
&\ls (1+\pi/\sqrt\kappa)\cdot  d_X(u(x),u(y_t)) \ls (1+\pi/\sqrt\kappa)\cdot \ell_0\cdot L_t(x) \\
& :=C_1(\ell_0,\kappa)L_t(x),
\end{split}
\end{equation}
where we have used
$$d_X(u(x),u(y_t))\ls d_X(u(x),Q_0)+d_X(Q_0,u(y_t))<\pi/\sqrt\kappa$$
Noticing that $f_t\ls0$, we get $D_t\gs\frac{L^p_t(x)}{pt^{p-1}}$. By combining with (\ref{equ4.4}), we have
$$L_t\ls ( p\cdot C_1)^{1/(p-1)}\cdot t:=C_2(p,\ell_0,\kappa)\cdot t.$$
By using  (\ref{equ4.4}) again, we get $D_t\ls  C_1C_2\cdot t.$ At last, $f_t=L^p_t/(pt^{p-1})-D_t\gs-D_t\gs -   C_1C_2\cdot t.$
The proof is finished.
\end{proof}

The main result in this subsection is the following elliptic inequality for $f_t$.
\begin{prop}\label{prop4.2} Assume that $\kappa =1$, $p\in[2,\infty)$ and   ${\rm diam}\big(u(B_{2R})\big)< \pi/2$, where ${\rm diam}\big(u(B_{2R})\big)$  is the diameter of the  $u(B_{2R})$.
   Given any $\varepsilon>0$,
we have that, for any  $t>0$  sufficiently small, the  inequality holds
\begin{equation}\label{equ4.5}
\Delta f_t(x) \ls \frac{K}{t^{p-1}}\cdot L^p_t(x)+ (1+\varepsilon)\cdot e_u(x) D_t(x),\quad {\rm on}\ \ B_R,
\end{equation}
 in the sense of distributions.
\end{prop}

In order to prove this lemma, we need the following:
\begin{lem}[Perturbation Lemma]\label{lem4.3}
Let $U\subset M$ be a convex domain of $M$ and let  $h\in W^{1,2}(U )\cap C(U)$ satisfy $\Delta  h\ls \lambda$  on $U$ for some constant $\lambda\in\R$. Assume  that point $\hat x\in U$ is one of   minimum points of function $h$ on $U$. Assume a subset $A\subset U$ has full measure.

Then   for any $r,\delta>0$ sufficiently small, (they are smaller than a constant $\delta_0$ depending on the bounds of sectional curvature on $U$,)
 there exists a smooth function $\phi$ on a neighborhood of $\hat x$, $B_{r_0}(\hat x)$, such that
$h+\phi$ has a local minimum point in $B_{r}(\hat x)\cap A$ and that
$$|\phi|+|\nabla\phi|+|{\rm Hess}(\phi)|\ls \delta,\qquad \forall\ x\in B_{r_0}(\hat x).$$
\end{lem}
\begin{proof}This comes from a slight extension of the classical  Jensen's lemma \cite[Lemma A.3]{c-i-lions92}. We will give the details of the proof in the Appendix A.
\end{proof}

\begin{proof}[Proof of Proposition \ref{prop4.2}] Denote  $\rho_0:={\rm diam}\big(u(B_{2R})\big)$.
It suffices to prove the following claim.\\
\textbf{Claim:} \emph{There exists $\bar t=\bar t(p,\varepsilon,t_*,\pi/2-\rho_0)$ such that for each $t\in(0,\bar t)$, the function $f_{t}(\cdot)$ satisfies
\begin{equation*}
\Delta  f_t(x) \ls  \frac{K}{t^{p-1}}\cdot L^p_t(x)+(1+\varepsilon)e_u(x) D_t(x)+\theta  \quad   on \ \  B_R
\end{equation*}
for any $\theta\in(0,1)$, in the sense of distributions.}

We shall prove this \textbf{Claim} by a contradiction argument.
 Suppose that the \textbf{Claim} fails for some sufficiently small $t\in (0,t_*)$ and some  $\theta_0\in(0,1)$.  According to the maximum principle,
 there exists a domain  $U\subset\subset B_R$
 such that the function $f_{t}(\cdot)-v(\cdot)$ satisfies
$$\min_{x\in U}\big(f_{t}(x)-v(x)\big)<0=\min_{x\in\partial U}\big(f_{t}(x)-v(x)\big), $$
where $v$ is the (unique) solution of the Dirichlet problem
\begin{equation*}
\begin{cases}
\Delta v&= \frac{K}{t^{p-1}}\cdot L^p_t(x)+(1+\varepsilon)e_u(x) D_t(x)+\theta_0 \quad  {\rm in}\ \ U\\
v&=f_{t} \quad {\rm on}\ \ \partial U.
\end{cases}
\end{equation*}
This means that $f_{t}(\cdot)-v(\cdot)$ has a \emph{strict minimum} in the interior of $U$.

Let us define a function $H(x,y)$ on $U\times B_{2R}$, by
$$H(x,y):=\frac{|xy|^p}{pt^{p-1}}- F\big(u(x),u(y)\big)-v(x).$$
Let $\bar x\in U$ be a minimum of $f_{t}(\cdot)-v$ on $U$, and let $\bar y:=y_t(\bar x)\in S_t(\bar x)$ such that $L_t(\bar x)=|\bar x\bar y|$.
By the definition of $f_t$, we conclude that $H(x,y)$ has a minimum at $(\bar x,\bar y)$.

Let  $\mathscr A\subset U\times B_{2R}$ be the set of all points $(x^o,y^o)\in U\times B_{2R}$ satisfying the following two properties:\\
\indent 1) $x^o\not=y^o$, and $x^o\not \in \mathscr N$, $y^o\not \in \mathscr N$, where  the set $\mathscr N$ is given as in Lemma \ref{lem3.3};  \\
\indent 2) point $x^o$ is a Lebesgue point of $\frac{K}{t^{p-1}}\cdot L^p_t(x)+(1+\varepsilon)e_u(x) D_t(x)$.\\
It is clear that $(U\times B_{2R})\backslash \mathscr A$ has zero measure.

Noting that the function $f(t)=2\sin t+4 \sin^2t$ satisfies $f'(t)>0,f''(t)>0$ for $t<\pi/6$ and (seeing the proof of Corollary \ref{cor2.6})
$$\Delta^{(2)}d_X(u(x),u(y))\gs0,$$
we have $\Delta^{(2)}F(x,y)\gs0$, where $\Delta^{(2)}$ is the Laplace-Beltrami operator on the product manifold $M\times M$. The Laplacian comparison theorem on product space $M\times M$ implies that
$\Delta^{(2)}(|xy|^2)\ls C_{n,K,{\rm diam}(U)}$ for some constant $C_{n,K,{\rm diam}(U)}>0$. By using the assumption   $p\gs 2$, we obtain $\Delta^{(2)}(|xy|^p)\ls C_{n,K,p,{\rm diam}(U)}$.
 Then, by Lemma \ref{lem4.3}, we conclude that,  for any sufficiently small $\delta>0$,  there exists a smooth  function
$\gamma_\delta(x,y)$ such that $|\gamma_\delta|+|\nabla \gamma_\delta|+|{\rm Hess}\gamma_\delta|\ls \delta$ and  that the function
$$H_1(x,y)=H(x,y)+\gamma_\delta(x,y)=\frac{|xy|^p}{pt^{p-1}}- F\big(u(x),u(y)\big)-v(x)+\gamma_\delta(x,y)$$
has a minimal point $(x^o,y^o)\in \mathscr A$ with
\begin{equation}\label{equ4.6}
d^2_{X\times X}\big((\bar x,\bar y),(x^o,y^o)\big)=|\bar x x^o|^2+|\bar yy^o|^2<\delta^2.
\end{equation}

Let $\sigma:[0,|x^oy^o|]\to M$ be a shortest geodesic with $\sigma(0)=x^o$ and $\sigma(|x^oy^o|)=y^o$ and let $T_{\sigma(t)}: T_{x^o}M\to T_{\sigma(t)}M$ be the parallel transport along  $\sigma(t)$.  Denote by $T:=T_{y^o}$. We want to  consider the asymptotic behavior of
\begin{equation}\label{equ4.7}
\begin{split}
I(\varepsilon_j):&=\!\!\int_{B_{\varepsilon_j}(0)\subset T_{x^o}M=\mathbb R^n}\Big[H_1\big(\exp_{x^o}(\eta),\exp_{y^o}(T\eta)\big)-H_1(x^o,y^o)\Big]{\rm d}\eta\\
&=I_1(\varepsilon_j)-I_2(\varepsilon_j)-I_3(\varepsilon_j)+I_4(\varepsilon_j),
\end{split}
\end{equation}
where the sequence $\{\epsilon_j\}$ is given in Lemma \ref{lem3.3} and
\begin{equation*}
\begin{split}
I_1(\varepsilon_j)&:=\frac{1}{pt^{p-1}}\cdot
\int_{B_{\varepsilon_j}(0)}\Big(|\exp_{x^o}( \eta)\ \exp_{y^o}( T\eta)|^p-|x^oy^o|^p\Big){\rm d}\eta,\\
I_2(\varepsilon_j)&:=
\int_{B_{\varepsilon_j}(0)}\!\Big(F\big(\exp_{x^o}( \eta),\exp_{y^o}(T\eta)\big)\!-\!F\big(x^o,y^o\big)\!\Big){\rm d}\eta,\\
I_3(\varepsilon_j)&:=
\int_{B_{\varepsilon_j}(0)}\Big(v(\exp_{x^o}( \eta))-v(x^o)\Big)dH^n(\eta),\\
I_4(\varepsilon_j)&:=
\int_{B_{\varepsilon_j}(0)}\Big( \gamma_\delta\big(\exp_{x^o}(\eta),\exp_{y^o}(T \eta))\big)-\gamma_\delta(x^o,y^o)\Big){\rm d}\eta.
 \end{split}
 \end{equation*}
The minimal property of point $(x^o,y^o)$ implies that
\begin{equation}\label{equ4.8}
 I(\varepsilon_j)\gs0.
\end{equation}
We need to estimate $I_1,I_2,I_3$ and $I_4$. \\

\noindent (i) \emph{The estimate of} $I_1$ and $I_4$.

Let $T$ be the parallel transportation,  the first and the second variation of arc-length implies that
\begin{equation*}
\begin{split}
|\exp_{x^o}(\eta)\ \exp_{y^o}(T\eta)|-|x^oy^o|\ls \frac{\epsilon^2}{2}\int_0^{|x^oy^o|} -R\big(\sigma'(t),T_{\gamma(t)}\eta, \sigma'(t),T_{\gamma(t)}\eta\big){\rm d}t+o(\epsilon^2)\\
\end{split}
\end{equation*}
for all $\eta\in B_\epsilon(0)$, where $R(\cdot,\cdot,\cdot,\cdot)$ is the Riemannian curvature tensor.  By taking $\epsilon=\varepsilon_j$ and using the fact that $a\ls b+\delta$ implies $a^p\ls b^p+p\delta b^{p-1}+o(\delta)$ as $\delta\to0$, and then integrating over $B_{\varepsilon_j}(0)$ we obtain
\begin{equation}\label{equ4.9}
  \begin{split}
I_1(\varepsilon_j)\ls\  \frac{1}{pt^{p-1}}\cdot\frac{pK\cdot \omega_{n-1}}{2n(n+2)} \cdot|x^oy^o|^p\cdot\varepsilon_j^{n+2} +\ o(\varepsilon_j^{n+2})
\end{split}
\end{equation}
for any $j\in \mathbb N$, where we have used $Ric(\sigma',\sigma')\gs -K$ and the Fatou's Lemma (since  $|\exp_{x^o}(\eta)\ \exp_{y^o}(T\eta)|-|x^oy^o|\ls C\epsilon^2$ for some constant $C$ depending on the sectional curvature on $B_{10|x^oy^o|}(x^o)$). 

Since $\gamma_\delta$ is smooth and that $|{\rm Hess}\gamma_\delta|\ls \delta$, it is easy to check that
\begin{equation}\label{equ4.10}
I_4(\varepsilon_j)  \ls C(n)\cdot \delta\cdot\varepsilon_j^{n+2}+ o(\varepsilon_j^{n+2}),
\end{equation}
for any $j\in \mathbb N$, and for some constant $C(n)>0$.

\noindent (ii) \emph{The estimate of} $I_2$.

We put
$$P=u\big(\exp_{x^o}(\eta)\big), \quad Q=u(x^o), \quad W=u(y^o) \quad {\rm and}\quad S=u\big(\exp_{y^o}(T\eta)\big),$$
and
\begin{equation}\label{equ4.11}
l_0:=2\sin\frac{d_{QW}}{2},\qquad   l_1:=2\cos\frac{d_{QW}}{2},\qquad \alpha=\frac{1}{1+2 l_0}\in (0,1).
\end{equation}
  Denote the midpoint of $Q,W$ by $Q_m$. Note that $\frac{1-\alpha}{2}=\alpha\cdot l_0$.  Then, by Lemma \ref{lem2.3} we have that for any $\beta>0$
\begin{equation}\label{equ4.12}
\begin{split}
\alpha l_0\cdot&  \Big(F(x^o,y^o)-F(\exp_{x^o}(\eta), \exp_{y^o}(T\eta))\Big) \\
 &=\alpha l_0\Big((2\sin \frac{d_{QR}}{2})^2-(2\sin \frac{d_{PS}}{2})^2\Big)+\alpha l_0\Big(2\sin \frac{d_{QR}}{2}-2\sin \frac{d_{PS}}{2}\Big) \\
&\ls\Big[1-\frac{1-\alpha}{2}(1-\frac{1}{\beta}) \Big](2\sin\frac{d_{PQ}}{2})^2+l_1\cdot\Big(\cos d_{PQ_m}-\cos d_{QQ_m}\Big)\\
 &\quad  +\Big[1-\frac{1-\alpha}{2}(1- \beta ) \Big](2\sin\frac{d_{RS}}{2})^2+l_1\cdot\Big(\cos d_{SQ_m}-\cos d_{RQ_m}\Big)\\
 &\ls\big[  w_{a_1,b,Q_m,x^o}(\exp_{x^o}(\eta))- w_{a_1,b,Q_m,x^o}(x^o)\big]\\
 &\quad+ \big[  w_{a_2,b,Q_m,y^o}(\exp_{y^o}(T\eta))- w_{a_2,b,Q_m,y^o}(y^o)\big],
 \end{split}
 \end{equation}
 where the function  $w_{a,b,Q_m,x^o}$ is given in Lemma \ref{lem3.3} with the constants
\begin{equation}\label{equ4.13}
a_1:= 1-\frac{1-\alpha}{2}(1-\frac{1}{\beta}),\quad b:=l_1,\quad a_2:= 1-\frac{1-\alpha}{2}(1- \beta ),
\end{equation}
 and we have used $2\sin(t/2)\ls t$ for any $t\in(0,\pi)$. From $\rho_0={\rm diam}(u(B_{2R}))$ we have
 $u(B_{2R})\subset \overline{B_{\rho_0}(u(x_0))}\cap \overline{B_{\rho_0}(u(y_0))}$. By the  assumption $\rho_0<\pi/2$ and that $X$ is a $CAT(1)$-space, we obtain
 $ u(B_{2R})\subset \overline{B_{\rho_0}(Q_m)}$.
Integrating over $B_{\varepsilon_j}(x^o)$,  Lemma  \ref{lem3.3} implies  that
\begin{equation*}
\begin{split}
\alpha l_0\cdot&\int_{B_{\varepsilon_j}(x^o)} \Big(F(x^o,y^o)-F(\exp_{x^o}(\eta), \exp_{y^o}(T\eta))\Big){\rm d}\eta\\
&\ls\Big[ 2a_1-b\cdot \cos\big(d_X(u(x^o),Q_m)\big)\Big]\cdot \frac{\omega_{n-1}\cdot e_u(x^o)}{2n(n+2)}\cdot\epsilon_j^{n+2}+\\
&\ \quad+\Big[ 2a_2-b\cdot \cos\big(d_X(u(y^o),Q_m)\big)\Big]\cdot \frac{\omega_{n-1}\cdot e_u(y^o)}{2n(n+2)}\cdot\epsilon_j^{n+2}+o(\epsilon_j^{n+2}).\\
\end{split}
\end{equation*}
Noticing that $\cos d_X\big(Q_m,u(x^o)\big)=l_1/2$ and that $1-l_1^2/4=l_0^2/4$, we choose $\beta $ such that
\begin{equation}\label{equ4.14}
 a_2=\frac{l^2_1}{4}\qquad \big(\Longleftrightarrow  \beta =1-\frac{l_0(1+2l_0)}{4}\big),
\end{equation}
where we have used $\alpha=1/(1+2l_0)$. Notice that $\beta>0$ provided $l_0\ls1$.
Then we have
\begin{equation}\label{equ4.15}
\begin{split}
-I_2(\varepsilon_j)&=\int_{B_{\varepsilon_j}(0)} \Big(F(x^o,y^o)-F(\exp_{x^o}(\eta), \exp_{y^o}(T\eta))\Big){\rm d}\eta\\
 &\ls\frac{ a_1- l^2_1/4}{\alpha l_0}\cdot \frac{\omega_{n-1}e_u(x^o)}{n(n+2)}\cdot\epsilon_j^{n+2}+ +o(\varepsilon^{n+2}_j).
\end{split}
\end{equation}
From $1-\beta=\frac{l_0}{4 \alpha}$ and $\frac{1}{\alpha}=1+2 l_0$, we have, if $l_0\ls \varepsilon/6$, that
\begin{equation*}
\begin{split}
\frac{ a_1- l^2_1/4}{\alpha l_0}&=\frac{1-\frac{1-\alpha}{2}(1-\frac 1 \beta)-l^2_1/4}{\alpha l_0}\\
&=l_0(1+2l_0)\Big(\frac{1}{4}+\frac{1}{4-l_0(1+2l_0)}\Big)\\
&\ls \frac{l_0}{2}(1+\varepsilon).
\end{split}
\end{equation*}
When  both $t$ and $\delta$ are small enough, the combination of Eq.(\ref{equ4.6}) and Lemma \ref{lem4.1}(3) implies that $l_0\ls \varepsilon/6$.  Therefore, we by (\ref{equ4.15}) get that
\begin{equation*}
\begin{split}
-I_2(\varepsilon_j)\ls(1+\varepsilon)\frac{l_0}{2}\cdot \frac{\omega_{n-1}}{n(n+2)}e_u(x^o)\cdot\varepsilon_j^{n+2}+o(\varepsilon^{n+2}_j).
\end{split}
\end{equation*}

\noindent (iii) \emph{The estimate of} $I_3$.

By Corollary \ref{cor3.2} and the definition of $v$, we have
\begin{equation*}
-I_3(\varepsilon_j)
\ls  \Big(\frac{-K}{t^{p-1}}\cdot L^p_t(x^o)-  (1+\varepsilon)  e_u (x^o) D_t(x^o)-  \theta_0 \Big)\cdot\frac{\omega_{n-1}}{2n(n+2)} \cdot \varepsilon_j^{n+2}+o(\varepsilon_j^{n+2}).
\end{equation*}

By combining  these estimates for $I_1, I_2,  I_3$ and $I_4$, we have
\begin{equation}\label{equ4.16}
\begin{split}
&\frac{K}{t^{p-1}} \Big(|x^oy^o|^p-L^p_{t}(x^o)\Big)+ (1+\varepsilon)  e_u(x^o)\Big(2\sin \frac{d_X(u(x^o),u(y^o))}{2}-D_t(x^o)\Big)\\
& - \theta_0 +C(n)\frac{2n(n+2)}{\omega_{n-1}} \delta\gs0.
\end{split}
\end{equation}
 Eq.(\ref{equ4.6}) implies that $(x^o,y^o)$ converge to $(\bar x,\bar y)$ as $\delta\to0$. We   by the lower semi-continuity of
$L_{t}$ and $D_t$ have that
$$\liminf_{\delta\to0} L_{t}(x^o)\gs L_{t}(\bar x)=\lim_{\delta\to0}|x^oy^o| $$
and
$$\liminf_{\delta\to0} D_{t}(x^o)\gs D_{t}(\bar x)=F(\bar x,\bar y)= \lim_{\delta\to0}2\sin \frac{d_X(u(x^o),u(y^o))}{2}. $$
A contradiction appears in (\ref{equ4.16}) when we take   $\delta\to0$ (noticing that  $K\gs0$). The proof is finished.
\end{proof}

\subsection{The Bochner inequality}$ \ $

We will prove Theorem \ref{bochner} in this subsection.

Let $p\in(1,\infty)$ and let $f_t$ be the auxiliary functions defined as in (\ref{equ4.1}) in the previous subsection,  on a ball $B_R$ with $B_{2R}\subset\subset \Omega$.

\begin{lem}\label{lem4.4} (i) Let $q\in(1,\infty)$ with $1/q+1/p=1$. For any $x\in B_R$, we   have
\begin{equation}\label{equ4.17}
\liminf_{t\to0}\frac{f_t(x)}{t}\gs -\frac{1}{q}{\rm Lip}^qu(x).
\end{equation}
(ii)   If $u$ is metrically differentiable  at $x$, then we have
\begin{equation}\label{equ4.18}
\lim_{t\to0^+}\frac{f_t(x)}{t}= -\frac{G^q_u(x)}{q}
\end{equation}
and
\begin{equation}\label{equ4.19}
\lim_{t\to0^+}\frac{L_t(x)}{t}=G_u^{q/p}(x),\qquad\quad  \lim_{t\to0^+}\frac{D_t(x)}{t}=G_u^q(x).
\end{equation}
\end{lem}

\begin{proof} (i) By the basic inequality $a^p/p-ab\gs -b^q/q$ for any $a,b\in\mathbb R$, we have
\begin{equation*}
\frac{1}{p}\cdot\frac{|xy|^p}{ t^p}-\frac{F(x,y)}{t}\gs -\frac{1}{q}\Big(\frac{F(x,y)}{|xy|}\Big)^q,\qquad\forall\ x,y\in B_{2R}.
\end{equation*}
Taking $y_t\in S_t(x)$ with $|xy_t|=L_t(x)$, we obtain from the definition of $f_t$ that
\begin{equation}\label{equ4.20}
\begin{split}
\liminf_{t\to0}\frac{f_t}{t}\gs&-\frac{1}{q}\limsup_{y_t\to x}\Big(\frac{F(x,y_t)}{|xy_t|}\Big)^q=-\frac{1}{q}\limsup_{t\to 0}\Big(\frac{D_t(x)}{L_t(x)}\Big)^q\\
&\gs -\frac{1}{q}{\rm Lip}^qu(x),
\end{split}
\end{equation}
where we have used $\lim_{y\to x} F(x,y)/d_X(u(x),u(y))=1.$ This proves (\ref{equ4.17}).

(ii) Let $u$ be  metrically differentiable at $x$. Take a unit vector $\xi\in T_xM$ such that
$$G_u(x)= mdu_x(\xi).$$
For each small $t>0$, we put $y_{t,x}:=\exp_x(tG_u^{q/p}\cdot\xi)$. Then
$$|xy_{t,x}|=t\cdot G_u^{q/p}(x)$$
and
\begin{equation*}\begin{split}
  d_X\big(u(x),u(y_{t,x})\big)&=  |xy_{t,x}|\cdot  mdu_x(\xi)+o(|xy_{t,x}|)\\
  &=t\cdot G^{q/p+1}_u(x)+o(t)= t\cdot G^{q}_u(x)+o(t),
  \end{split}\end{equation*}
as $t\to0$. Thus,
 by the definition of $f_t$, we obtain
 $$ \frac{f_t(x)}{t}\ls \frac{|xy_{t,x}|^p}{pt^p}-\frac{F(x,y_{t,x})}{t} =\frac{G^q_u(x)}{p}-G^q_u(x) +o(1)=- \frac{G^q_u(x)}{q}  +o(1).$$
as $t\to0$. That is
\begin{equation}\label{equ4.21}
\limsup_{t\to0^+}\frac{f_t(x)}{t}\ls -\frac{G^q_u(x)}{q}
\end{equation}
 Recall that $G_u(x)={\rm Lip}u(x)$ by Lemma \ref{lem2.9}, the combination of (\ref{equ4.17}) and (\ref{equ4.21}) yields (\ref{equ4.18}).

Combining (\ref{equ4.18}) and  (\ref{equ4.20}) together gives
$$\limsup_{t\to 0}\Big(\frac{D_t(x)}{L_t(x)}\Big)^q={\rm Lip}^qu(x).$$
On the other hand,  notice that
$$-\frac{f_t(x)}{t}=-\frac{L^p_t(x)}{pt^p}+\frac{D_t(x)}{L_t(x)}\cdot \frac{L_t(x)}{t}\ls\frac 1 q\Big(\frac{D_t(x)}{L_t(x)}\Big)^q.$$
Thus, we get
$$\liminf_{t\to0} \frac{D_t(x)}{L_t(x)} \gs {\rm Lip}u(x).$$
Therefore, we obtain
 $$\lim_{t\to0} \frac{D_t(x)}{L_t(x)} = {\rm Lip}u(x).$$
By using $f_t/t=\frac{L^p_t}{pt^p}-\frac{D_t}{L_t}\cdot \frac{L_t}{t}$ again, it follows $\lim_{t\to0} \frac{L_t(x)}{t} = {\rm Lip}^{\frac{1}{p-1}}u(x)={\rm Lip}^{q/p}u(x),$ and then $\lim_{t\to0} \frac{D_t(x)}{t} = {\rm Lip}^{1+q/p}u(x)={\rm Lip}^{q}u(x).$
 The proof is finished.
\end{proof}

Now we are in the place to prove Theorem \ref{bochner}.
\begin{proof}[Proof of Theorem  \ref{bochner}:] We have known that ${\rm Lip}\in L^\infty_{\rm loc}(\Omega)$ from Theorem 1.3. By a rescaling argument of the target space, we can assume that $(X,d_X)$ is a $CAT(1)$-space. In this case,
since ${\rm Lip}u=G_u$ for almost all $x\in \Omega$, we need   to prove
$G_u\in W^{1,2}_{\rm loc}(\Omega)$  and that
\begin{equation}\label{equ4.22}
\Delta G_u\gs -K G_u-  e_u\cdot G_u
\end{equation}
on  $\Omega$, in the sense of distributions. It suffices to show that: for any $o\in \Omega$, there exists a neighborhood $B_{R}(o)$ with $B_{2R}(o) \subset \Omega$ such that $G_u\in W^{1,2}_{\rm loc}(B_{R}(o))$  and that (\ref{equ4.22}) holds
on  $B_{R}(o)$, in the sense of distributions.

Since $u$ is continuous (from Theorem \ref{thm2.4}), for any $o\in \Omega$, there exists a neighborhood $B_{R}(o)$ such that
 ${\rm Image}(u(B_{2R}(o))\big)\subset B_{\pi/4}\big(u(o)\big).$ In particular,  the triangle inequality yields ${\rm diam}u(B_{2R}(o))<\pi/2.$ Fix such a neighborhood $B_R =B_R(o)$.

Fix any $p\in[2,\infty)$. From Lemma \ref{lem4.1}(3), we get $\Delta f_t/t\ls C_{\ell_0}$ on $B_{3R/2}$ for any $t\in(0,t_*)$, where the constant $C_{\ell_0}$ is uniformly with respect to $t\in (0,t_*)$.   By combining Lemma \ref{lem4.4}  and Proposition   \ref{prop4.2} together, we have, for any $\varepsilon>0$ that $G^q_u/q\in W^{1,2}_{\rm loc}(B_{3R/2})$  and that
$$\Delta(G^q_u/q)\gs- K G^q_u- (1+\varepsilon)\cdot e_u\cdot G_u^q$$
on  $B_{3R/2}$, in the sense of distributions, where $q=p/(p-1)\in(1,2]$.
  From the arbitrariness of $\varepsilon,$ we conclude that
\begin{equation}\label{equ4.23}
\Delta (G^q_u/q)\gs- K G^q_u-  e_u\cdot G^q_u
\end{equation}
on  $B_{3R/2}$, in the sense of distributions, where $q=p/(p-1)\in(1,2]$.

 In order to take the limit of (\ref{equ4.23}) as $q\to1$, we want to show that the energies $\|\nabla  (G_u^q/q) \|_{L^2(B_{R})}$ are bounded uniformly  with respect to $q$. By the local Lipschitz continuity of $u$, there exists a constant $C_1>1$ such that $G_u, e_u\ls C_1$ on $B_{3R/2}$. Hence, we have
$$\Delta( G_u^q/q)\gs -K\cdot C_1^q-C_1^{q+1}\gs- K\cdot C_1^2-C_1^{3} $$
 on  $B_{3R/2}$, in the sense of distributions, where we have used $K\gs0 $ and $q\in(1,2].$ By applying Caccioppoli inequality, we conclude that the energies $\|\nabla  (G_u^q/q) \|_{L^2(B_{R})}$ are bounded uniformly. Hence, there exists a sequence $\{q_j\}_{j\in\mathbb N}$ with $q_j\in (1,2]$ and $q_j\searrow 1$ such that $\Delta(G_u^{q_j}/q_j)\rightharpoonup \Delta  G_u$. Now, by letting $q_j\to1$ in (\ref{equ4.23}), we conclude that
$G_u\in W^{1,2}(B_R)$  and (\ref{equ4.22}). The proof is finished.
 \end{proof}

\section{Yau's  gradient estimates}

We will continue to assume that  $\Omega$ is a smooth domain of an $n$-dimensional Riemannian manifold $(M,g)$ with $Ric\gs -K$ for some  $K\gs0$,  and that $(X,d_X)$ is a $CAT(\kappa)$-space for some $\kappa\gs0$.
Let $u$ be a harmonic map  from $\Omega$ to $X$. Assume that its image $u(\Omega)\subset X$ is contained in a ball  with radius $<\frac{\pi}{2\sqrt \kappa}$.

When the target space has non-positive curvature, we have the following a consequence of the Bochner inequality, Theorem \ref{bochner}.
\begin{lem}\label{lem5.1}
Let $\kappa=0$. Suppose that $B_{2R}(x_0)\subset\subset \Omega$ and  that $u(B_R(x_0))\subset B_\rho(Q_0)$ for some $\rho>0$ and $Q_0\in X$. We put
$$h=2\rho^2- d_X^2\big(Q_0, u(x)\big)\quad {\rm and}\quad F=\frac{{\rm Lip} u}{h}.$$
Then $F\in W^{1,2}\cap L^\infty(B_R(x_0))$ and
\begin{equation}\label{equ5.1}
\Delta F+2\ip{\nabla F}{\nabla\log h}\gs C \rho^2\cdot F^3-K\cdot F
\end{equation}
in the sense of distribution, where $C=C_{n,\sqrt KR}.$
\end{lem}

\begin{proof}
From Theorem \ref{bochner}, we have ${\rm Lip}u\in  W^{1,2}\cap L^\infty(B_R(x_0))$. Noticing that $h$ is Lipschitz continuous and $ h\gs\rho^2,$  we obtain $F\in W^{1,2}\cap L^\infty(B_R(x_0))$.

 By applying the Chain rule to ${\rm Lip}u=hF$, we have
\begin{equation*}
h\cdot\Delta F +2\ip{\nabla F}{\nabla h}+F\cdot \Delta h= \Delta {\rm Lip}u.
\end{equation*}
Multiplying  both sides of this inequality by $h^{-1}$ and substituting (\ref{equ1.7}) and then $-\Delta h\gs 2e_u$ (see Lemma \ref{lem2.5}), we get
\begin{equation}\label{equ5.2}
\begin{split}
\Delta F +2\ip{\nabla F}{\nabla\log h} = -F\cdot \frac{\Delta h}{h}+\frac{\Delta {\rm Lip}u}{h}
 \gs F\cdot \frac{2 e_u}{h}-K F.
\end{split}
\end{equation}
  in the sense of distributions. As Corollary \ref{cor2.6}, we have
\begin{equation*}
\begin{split}
\Delta F +2\ip{\nabla F}{\nabla\log h}&\gs C_{n,\sqrt KR}\frac{F}{h}\cdot {\rm Lip}^2u -KF = C_{n,\sqrt KR}hF^3-KF\\
&\gs C_{n,\sqrt KR}\cdot\rho^2\cdot F^3-KF
\end{split}
\end{equation*}
 where we have used $h\gs \rho^2$ again.
 The proof is finished.
\end{proof}

Similarly, in the case of where the target is a $CAT(1)$-space, we have the following property.

\begin{lem}\label{lem5.2}
Let $\kappa=1$. Suppose that $B_{2R}(x_0)\subset\subset \Omega$ and  that $u(B_R(x_0))\subset B_\rho(Q_0)$ for some $\rho<\pi/2$ and $Q_0\in X$. We put
$\rho_0=\frac{\rho+\pi/2}{2}=\rho+\frac{\pi/2-\rho}{2}$ and
$$h_1= \cos  d_X\big(Q_0, u(x)\big) -\cos \rho_0\quad {\rm and}\quad F=\frac{{\rm Lip} u}{h_1}.$$
Then $F\in W^{1,2}\cap L^\infty(B_R(x_0))$ and
\begin{equation}\label{equ5.3}
\Delta F+2\ip{\nabla F}{\nabla\log h_1}\gs C'\cdot F^3-KF
\end{equation}
in the sense of distribution, where $C'=C'(n,\sqrt KR,\pi/2-\rho)$ is a constant depends on $n,\sqrt KR$ and $\pi/2-\rho.$
\end{lem}

\begin{proof} It is easy to see that $h_1$  is Lipschitz continuous and
$$  h_1\gs \cos\rho-\cos\rho_0:=C'_{1}>0$$ for some positive number $C'_1$ depends only on $\pi/2-\rho.$
From Lemma \ref{lem2.5}, we have
$$-\Delta h_1\gs \cos d_X^2\big(Q_0, u(x)\big)\cdot e_u\gs \cos\rho \cdot e_u.$$
 The remain of the proof, by a  similar  argument as in Lemma \ref{lem5.1}, we have
\begin{equation}\label{equ5.4}
\Delta F +2\ip{\nabla F}{\nabla\log h_1}\gs C_{n,\sqrt KR}\cdot C'_1\cdot\cos\rho \cdot F^3-KF
\end{equation}
in the sense of distributions.
\end{proof}

In general, the function ${\rm Lip}u$ does not smooth, even  may not be continuous. It is difficult to employ the argument in \cite{cheng80} directly. So we will ues the following the approximating version of the maximum principle.
\begin{thm}\label{app-mp}{\rm (\cite[Theorem 1.4]{zz16-2})}
Let $f(x)\in W^{1,2}_{\rm loc}\cap L^\infty_{\rm loc}(\Omega)$  such that  $\Delta f $ is a signed Radon measure with $\Delta^{\rm sing}f\gs 0$, where
$\Delta f=\Delta^{\rm ac}f\cdot {\rm vol}_g+\Delta^{\rm sing}f$ is the Radon-Nikodym  decomposition  with respect to ${\rm vol}_g$.
Suppose that $f$ achieves one of its   strict maximum in  $\Omega$ in the sense that: there exists a neighborhood $U\subset\subset \Omega$ such that
\begin{equation*}
\sup_{U}f>\sup_{\Omega\backslash U}f.
\end{equation*}

Then, given any $w\in  W^{1,2}(\Omega)\cap L^\infty(\Omega)$, there exists a sequence of points $\{x_j\}_{j\in\mathbb N}\subset U$  such that they are the approximate continuity points of $\Delta^{\rm ac}f$ and $\ip{\nabla f}{\nabla w}$, and that
$$f(x_j)\gs \sup_\Omega f-1/j\qquad{\rm and }\qquad  \Delta^{\rm ac}f(x_j)+\ip{\nabla f}{\nabla w}(x_j)\ls 1/j.$$
Here and in the sequel, $\sup_{U}f$ means  ${\rm ess\!-\!sup}_{U}f$.
\end{thm}
\begin{proof}It was proved in \cite{zz16-2} in the setting of metric measure spaces with generalized Ricci curvature bounded from below. In particular, it holds for Riemannian manifolds with Ricci curvature bounded from below (with the Riemannian measure).
\end{proof}
The proof of  the Theorem \ref{thm1.4} and Theorem \ref{thm1.8} are both based on the following lemma.
\begin{lem}\label{lem5.4}
Let $B_R(x_0)\subset\Omega$ and  let  $F\in W^{1,2}_{\rm loc}\cap L^\infty_{\rm loc}(B_R(x_0))$  be a nonnegative function. Assume  that $F$ satisfies
\begin{equation}\label{equ5.5}
\Delta F+\ip{\nabla F}{\nabla v}\gs a_1F^3-a_2F,
\end{equation}
in the sense of distributions, where $v\in W^{1,2}_{\rm loc}\cap L^\infty_{\rm loc}(B_R(x_0))$ such that $|\nabla v|\ls a_3F$, and  the constants $a_1,a_3>0$ and $a_2\gs0.$
Then there exists a constant $C_{n,\sqrt KR}$ such that
$$\sup_{B_{R/2}(x_0)}F^2\ls  \frac{2a_2}{a_1} +\frac {C_{n,\sqrt KR}}{R^2}\bigg(\frac{1}{a_1}+\frac{a_3^2}{ a_1^2}\bigg).$$
\end{lem}

\begin{proof}
Fix  any a small  number $\delta$ such that  $0<\delta<  \frac 1 2\frac{\sup_{B_{R/2}}F}{\sup_{B_{3R/4}}F}$. Let us choose   $\eta(x)=\eta(r(x))$ to be a function of the distance $r$ to the fixed point $x_0$ with the following property that
$$  \delta\ls \eta \ls1\ \ {\rm on}\ \ B_{R},\qquad \eta =1\ \ {\rm on}\ \ B_{R/2},\qquad \eta =\delta\ \ {\rm on}\ \ B_{R}\backslash B_{3R/4},$$
and
$$ - \frac{C_1}{R} \ls  \eta'(r)  \ls 0\quad {\rm and}\quad  |\eta ''(r)|   \ls \frac{C_1}{R^2} \qquad \forall\ r\in(0,3R/4)$$
for some  universal constant $C_1$ (which is independent of  $n,K,R$). Then we have
\begin{equation}\label{eq5.6}
 |\nabla \eta | =  |\eta '||\nabla r| \ls  \frac{C_1}{R} \quad\ {\rm on}\ \    B_{3R/4},
\end{equation}
and, by the Laplacian comparison theorem,  that
\begin{equation}\label{eq5.7}
\begin{split}
 \Delta\eta &=\eta' \Delta r+\eta''|\nabla r|^2\gs-\frac{C_1}{R} \Big(\sqrt{(n-1)K}\coth\big(r\sqrt{K/(n-1)}\big)\Big)  - \frac{C_1}{R^2} \\
 &   \gs -\frac{C_1}{R} \Big(\sqrt{(n-1)K }+\frac{n-1}{R}\Big)  - \frac{C_1}{R^2} =- \frac{C_1\sqrt{(n-1)K}R+nC_1}{R^2} \gs -\frac{C_2}{R^2}
 \end{split}
 \end{equation}
on $  B_{3R/4}$, in the sense of distributions, where we have used that
$$\coth\big(r\sqrt{K/(n-1)}\big)\ls \coth\big(R\sqrt{K/(n-1)}\big)\ls  1+\frac{1}{R\sqrt{K/(n-1)}}. $$
Here and in the sequel of this proof, we denote $C_1,C_2,C_3, \cdots$  the various constants which depend only on $n$ and $\sqrt KR$.

Now we put $G=\eta F$.  
Then  $G$ is in $ W^{1,2} (B_{3R/4}) \cap L^\infty (B_{3R/4})$ and $G$ achieves one of its  strict maximum in  $B_{R/2}$ in the   sense of Theorem \ref{app-mp}.
\begin{equation*}
\begin{split}
 \Delta G +\ip{\nabla G}{\nabla v}&=\Delta \eta\cdot F+2\ip{\nabla \eta}{\nabla (G/\eta)}+\eta\cdot \Delta F+\eta\ip{\nabla F}{\nabla v}+F\ip{\nabla \eta}{\nabla v}\\
 &\gs \Delta\eta\cdot \frac G \eta+2\ip{\nabla \log\eta}{\nabla G}-2\frac{|\nabla \eta|^2}{\eta}\cdot\frac G \eta+\eta\big(\Delta F+\ip{\nabla F}{\nabla v}\big)+\frac G \eta\ip{\nabla \eta}{\nabla v}\\
  &\gs - \frac G \eta\frac{C_2}{R^2}+2\ip{\nabla \log\eta}{\nabla G} - \frac{G}{\eta^2}\frac{2C^2_1}{R^2}+\eta\cdot \big(a_1F^3-a_2F\big)+\frac G \eta\ip{\nabla \eta}{\nabla v}.
  \end{split}
 \end{equation*}
By setting $w=v-2\log\eta\in  W^{1,2}(B_{3R/4})\cap L^\infty (B_{3R/4})$ and using
$|\nabla v|\ls a_3F=a_3\frac G \eta,$
we have
\begin{equation*}
\begin{split}
 \Delta G +\ip{\nabla G}{\nabla w}& \gs -\frac{C_2}{R^2}\frac G \eta-\frac{2C^2_1}{R^2}\frac{G}{\eta^2} +\eta\big(a_1(G/\eta)^3-a_2(G/\eta)\big)-a_3\frac G \eta |\nabla\eta| \frac G \eta\\
 & \gs -\frac{C_2}{R^2}\frac G \eta-\frac{2C^2_1}{R^2}\frac{G}{\eta^2} +a_1   \frac{G^3}{\eta^2}-a_2G-a_3\frac{G^2}{ \eta^2}  \frac{C_1}{R}\\
 & \gs -\frac{C_2}{R^2}\frac G \eta\frac 1 \eta-\frac{2C^2_1}{R^2}\frac{G}{\eta^2} +a_1   \frac{G^3}{\eta^2}-a_2G \frac {1}{\eta^2}-a_3\frac{G^2}{ \eta^2}  \frac{C_1}{R}\\
 &  \gs \frac{G}{\eta^2}\Big[-\frac{C_2}{R^2} -\frac{2C^2_1}{R^2}  +a_1  G^2-a_2 -a_3  \frac{C_1}{R}G\Big],
  \end{split}
 \end{equation*}
 where we have used $G\gs0$,   $1/\eta\gs1$ and $a_2\gs0$. Let $C_3:=C_2+2C_1^2$. Substituting  $\frac{a_3C_1}{R}\cdot G\ls \frac{a_1}{2}G^2+\frac{1}{2a_1}(\frac{a_3C_1}{R})^2$ into the above inequality, we obtain
 \begin{equation}\label{eq5.8}
 \Delta G +\ip{\nabla G}{\nabla w}\gs \frac{G}{\eta^2}\Big[-\frac{C_3}{R^2}    +\frac{a_1 }{2} G^2-a_2 -\frac{C_1^2a_3^2}{2a_1 R^2}\Big]
 \end{equation}
 in the sense of distributions.
According to  Theorem \ref{app-mp},  there exit  a sequence $\{x_j\}_{j\in\mathbb N}$ such that, for each $j\in\mathbb N$,
\begin{equation*}
G_j:=G(x_j)\gs \sup_{B_{3R/2}}G-1/j
\end{equation*}
and that
\begin{equation*}
\frac{G_j}{\eta^2(x_j)} \Big[\frac{a_1 }{2} G_j^2-a_2 -\frac{C_3}{R^2} -\frac{C_1^2a_3^2}{2a_1 R^2}\Big]\ls \frac{1}{j}.
\end{equation*}
As $\eta\gs \delta>0$, by  letting $j\to\infty$, we have
\begin{equation*}
\sup_{B_{3R/4}}G^2=\lim_{j\to\infty} G_j^2\ls   \frac{2a_2}{a_1} +\frac{2C_3}{a_1R^2}+\frac{C_1^2a_3^2}{ a_1^2 R^2}.
 \end{equation*}
This yields
$$\sup_{B_{R/2}}F^2\ls \sup_{B_{3R/4}}G^2\ls  \frac{2a_2}{a_1} +\frac{2C_3}{a_1R^2}+\frac{C_1^2a_3^2}{ a_1^2 R^2}\ls \frac{2a_2}{a_1} +\frac {C_4}{R^2}\Big(\frac{1}{a_1}+\frac{a_3^2}{ a_1^2}\Big),$$
where $C_4:=\max\{2C_3,C_1^2\}.$ The proof is finished.
\end{proof}

Now we are in the place to show the main results.
\begin{proof}[Proof of Theorem \ref{thm1.4}] By applying Lemma 5.4 to (\ref{equ5.1}) in Lemma 5.1 with $v=2\log h$ and
$$a_1=C_5\rho^2,\quad a_2=K,\quad a_3=2\rho,$$
and noticing that
$$|\nabla v|=2\frac{|\nabla h|}{h}=2\frac{d_X \big(Q_0, u(x)\big)\cdot |\nabla d_X\big(Q_0, u(x)\big)|}{h}\ls 2\frac{\rho\cdot {\rm Lip} u}{h}=2\rho F,$$
We conclude that, for some constants $C_5,C_6,C_7$  depending only on $n$ and $\sqrt KR$, it holds
\begin{equation}\label{equ5.9}
\sup_{B_{R/2}(x_0)}\frac{{\rm Lip}^2 u}{\big(2\rho^2-d_X^2\big(Q_0, u(x)\big)\big)^2}\ls \frac{2K}{C_5\rho^2}+\frac{C_6}{R^2}\Big(\frac{1}{C_5\rho^2}+\frac{4\rho^2}{C_5^2\rho^4}\Big)\ls \frac{C_7}{\rho^2}\Big(K+\frac{1}{R^2}\Big).
\end{equation}
 This implies
\begin{equation*}
\begin{split}
\sup_{B_{R/2}(x_0)} {\rm Lip}^2 u & \ls \frac{C_7}{\rho^2}\Big(K+\frac{1}{R^2}\Big)\cdot \sup_{B_{R/2}(x_0)} \big(2\rho^2-d_X^2\big(Q_0, u(x)\big)\big)^2\\
& \ls  C_7 \frac{4\rho^2}{R^2} (KR^2+1 )=C_8\cdot \frac{\rho^2}{R^2}.
\end{split}
\end{equation*}
for some constant  $C_8$  depending only on $n$ and $\sqrt KR$.  The proof is finished.
\end{proof}

\begin{proof}[Proof of Theorem \ref{thm1.8}] By applying Lemma 5.4 to Lemma 5.2 with $v=2\log h$ and noticing that
$$|\nabla v|=2\frac{|\nabla h|}{h}=2\frac{\sin d_X\big(Q_0, u(x)\big)\cdot |\nabla d_X \big(Q_0, u(x)\big)|}{h}\ls 2\frac{ {\rm Lip} u}{h}=2  F,$$
and choosing
$$a_1=C'_1,\quad a_2=K,\quad a_3=2,$$
Here and in the sequel of this proof, we denote $C'_1,C'_2,C'_3, \cdots$  the various constants which depend only on $n, \sqrt KR$ and  $ \pi/2-\rho$.
We conclude that
$$\sup_{B_{R/2}(x_0)}\frac{{\rm Lip}^2 u}{\big( \cos d_X(Q_0, u)  -\cos \rho_0\big)^2}\ls \frac{2K}{C'_1}+\frac{C_2'}{R^2}\Big(\frac{1}{C'_1}+\frac{4}{(C'_1)^2}\Big)\ls C_3'\Big(K+\frac{1}{R^2}\Big),$$
 where $\rho_0=\pi/4+\rho/2$. By noticing that $ \cos (d_X\circ u)-\cos \rho_0\ls   \cos \rho$, this implies
\begin{equation*}
\sup_{B_{R/2}(x_0)} {\rm Lip}^2 u  \ls C_3'\cos\rho\Big(K+\frac{1}{R^2}\Big)\ls \frac{C_4'}{R^2}.
\end{equation*}
 The proof is finished.
\end{proof}

 \appendix
   \renewcommand{\appendixname}{Appendix~\Alph{section}}

       \section{An generalized Jensen's lemma and the proof of Lemma \ref{lem4.3}.}
We need a simple lemma for symmetric matrices as follows.
\begin{lem}\label{lem-a1}
Let $A=(a_{ij})_{n\times n},B=(b_{ij})_{n\times n}$ be two symmetric matrices. Assume that $B$ is nonnegative definite, and that $|a_{ij}-\delta_{ij}|\ls \frac{1}{2n^2}$ for any $1\ls i,j\ls n$, where $I=(\delta_{ij})_{n\times n}$ is the identity matrix. If ${\rm trace}(AB)\ls C$ for some $C>0$, then we have
$$|{\rm det}B|\ls (2C)^n.$$
\end{lem}
\begin{proof}
We put $\bar \mu=$ the maximum eigenvalue of $B$. Then by nonnegative definiteness of $B$ we have $\bar\mu\ls \|B\|:=(\sum_{i,j=1}^nb^2_{ij})^{1/2}\ls \sqrt{n}\bar\mu.$
 Hence we have
\begin{equation*}
\begin{split}
\bar\mu\ls {\rm trace}(B)&={\rm trace}\big((I-A)B\big)+  {\rm trace}\big(AB\big) \\
&\ls \|I-A\|\cdot \|B\| + {\rm trace}\big(AB\big) \\
&\ls \big[n^2\cdot \big(\frac 1{2n^2}\big)^2\big]^{1/2}\cdot \sqrt{n}\bar\mu +C=\frac{\bar \mu }{2\sqrt n} +C.
\end{split}
\end{equation*}
This implies that $\bar \mu\ls 2C$. At last, by the assumption that $B$ is nonnegative definite, we have
$0\ls {\rm det}B\ls \bar\mu^n\ls (2C)^n.$
The proof is finished.
\end{proof}

The following lemma is a slight extension of  Jensen's lemma (see, for example, \cite[Lemma A.3]{c-i-lions92}).
\begin{lem}\label{lem-a2}
Let $U\subset M$ be a convex domain of $M$ and let  $h\in W^{1,2}(U )\cap C(U)$ satisfy $\Delta  h\ls \lambda$  on $U$ for some constant $\lambda\in\R$. Assume  that point $\hat x\in U$ is a uniquely local minimum point of function $h$ on $U$. Let $\{y^j\}_{1\ls j\ls n}$    be a local geodesic coordinate system  around $\hat x$. For  any  $p=(p^1,\cdots,p^n)\in \mathbb R^n$, we set
$$h_{p}(x):=h(x) +\sum_{i,j=1}^n p^j  y^j(x).$$
Then   for any $r,\delta>0$ sufficiently small, (namely, they are smaller than a constant $\delta_0$ depending on the bounds of sectional curvature on $U$,)  the set
$$K=\big\{x\in B_{r}(\hat x):\ \exists  p\in B_\delta(0)\ {\rm  for\ which}\ h_{p}\ {has \ a \ local \ minimum\ at }\ x\big\}$$
has positive measure, where $B_\delta(0):=\{p\in \mathbb R^n: \ \|p\|<\delta,\ \forall\ 1\ls j\ls n\}$.
\end{lem}
\begin{proof}
Fixed any sufficiently small $r$, if $\delta$ is  small enough,  then for any  $p\in B_\delta(0)$, there exists a local minimum of  $h_{p}$  lying  in the interior of  $B_{r}(\hat x)$, since $h$ has the unique minimum at $\hat x$. We will split the proof into two steps, as in the argument of \cite[Lemma A.3]{c-i-lions92}.

(i) We assume for the moment that $h$ is of $C^2$  near $\hat x$.
Let  $(g_{ij})_{n\times n}$  be  the local Riemannian metric around $\hat x$ with respect to the coordinate system $\{y^j\}_{1\ls j\ls n}$. There exists a number $r_0>0$, depending on the curvature on $U$, such that, for all  $1\ls i,j,k\ls n,$
$$|g_{ij}(x)-\delta_{ij}|\ls \frac{1}{10n^2},\quad |\partial_kg_{ij}(x)|\ls 1,\qquad \forall\ x\in B_{r_0}(\hat x).$$
Fixed any  $r\in (0,r_0)$, it suffices to show that $K$ has positive measure with respect to the Euclidean measure on $\big(B_{r}(\hat x),\delta_{ij}\big)$.

Now we consider the elliptic operator
$$Lh:=\sum_{1\ls i,j\ls n}\partial_i\big(a_{ij} \partial_{j}h\big)\qquad {\rm with}\quad  a_{ij}=g^{ij}\sqrt{{\rm det}(g_{ij})}.$$
It is easily seen that there exists  a constant $C(n)$ such that for all $x\in B_{r}(\hat x)$ and for all $1\ls i,j,k\ls n$, they hold
 \begin{equation}
 |a_{ij}(x)-\delta_{ij}|\ls \frac{1}{4n^2},\quad |\partial_ka_{ij}|\ls C(n),\quad Lh\ls C(n)\cdot\lambda.
 \end{equation}
Since $D h+p=D h_{p}=0$ holds for every  minimum points  of $h_p$, we have
 $$D h(K)\supseteq B_\delta(0).$$
 Here $Dh$ (and the following $D^2h $) is the  (2-order) differential of $h$ with respect to the Euclidean metric on $B_{r}(\hat x)$.
 Moreover,  for every  minimum points  of $h_p$,  we have that $ D^2h_p=(\partial_i\partial_jh_p) $ is nonnegative definite and that
\begin{equation*}
\begin{split}
\sum_{1\ls i,j\ls n}a_{ij}\partial_i\partial_j h_p &=\sum_{1\ls i,j\ls n}\partial_i\big(a_{ij}\partial_j h_p\big)\qquad ({\rm since}\ \partial_j h_p=0)\\
&=Lh+L(\sum_{k=1}^np_k\cdot x_k)=Lh+\sum_{1\ls i,j\ls n}p_j\cdot\partial_ia_{ij}\\
&\ls C(n)\cdot\lambda+C(n)n^2\delta\ls C_1(n,\lambda),
\end{split}
\end{equation*}
for a constant $C_1>0$. By using Lemma \ref{lem-a1} for $B=D^2h_p$, we have $|{\rm det}D^2h|=|{\rm det}D^2h_p|\ls (2C_1)^n $ for all $x\in K$. Thus,
$$\mathcal{L}^n\big(B_\delta(0)\big)\ls \mathcal{L}^n\big(Dh(K)\big)\ls\int_K|{\rm det}D^2h|dx\ls \mathcal{L}^n(K)\cdot (2C_1)^n,$$
where $\mathcal{L}^n(K)$ is the Euclidean measure of $K$. This completes the proof for the smooth case.

(ii) In the general case, in which $h$ need not to be smooth, we will approximate it via heat flows. This is the reason that we have to assume that $U$ is convex.

Let $\{P_th\}_{t\gs0}$ be the heat flow with Neumann boundary value condition on $U$, with the initial data $P_0h=h$. It is clear that $P_th$ is smooth for any $t>0$. By maximum principle, we have
$$\Delta P_th=P_t\Delta h\ls \lambda,\qquad \forall\ t>0.$$
The corresponding set $K_t$ obey the above estimates in (i) for small $t>0$.  In particular, the   measure of $K_t$ is bounded from below by a constant $C(\delta,\lambda,n)>0$ uniformly on $t>0$.

At last, by using  the convexity of the  boundary of $U$ and that  the  curvature of $M$ is bounded on $U$, (in particular, the Ricci curvature on $U$ is bounded from below,) the Li-Yau gradient estimates for solutions of the heat flow implies that $P_th$ converge  uniformly to $h$ on $B_{r}(\hat x)\subset\subset U$. Notice that $K\supset \liminf_{t\to0}K_t$. The result now follows.
\end{proof}
Now the perturbation Lemma \ref{lem4.3} is a corollary as follows.
\begin{cor}\label{lem-a3}
Let $U\subset M$ be a convex domain of $M$ and let  $h\in W^{1,2}(U )\cap C(U)$ satisfy $\Delta  h\ls \lambda$  on $U$ for some constant $\lambda\in\R$. Assume  that point $\hat x\in U$ is one of   minimum points of function $h$ on $U$. Assume a subset $A\subset U$ has full measure.
Then   for any $r,\delta>0$ sufficiently small, there exists a smooth function $\phi$ on a neighborhood of $\hat x$, $B_{r_0}(\hat x)$, such that
$h+\phi$ has a local minimum point in $B_{r}(\hat x)\cap A$ and that
$$|\phi|+|\nabla\phi|+|{\rm Hess}(\phi)|\ls \delta,\qquad \forall\ x\in B_{r_0}(\hat x).$$
\end{cor}
\begin{proof}
Fix any $r,\delta$ sufficiently small. We put $h_1:=h+\delta|\hat xx|^2/(10n).$ Then $h_1$ has a unique minimum  at $\hat x$. Since $\Delta h_1\ls \Delta h+C(n,k_0)\cdot\delta$ by Laplacian comparison on $M$, the above Lemma implies that $h_1+\sum_{j=1}^np^jy^j$ has a local minimum at a point in $B_{r}(\hat x)\cap A$ and that $0\ls p^j\ls \delta/2,$ $1\ls j\ls n$.
Now, the function
$$\phi:=\delta|\hat xx|^2/(10n)+\sum_{j=1}^np^jy^j$$
defined on a coordinate neighborhood   $B_{r_0}(\hat x)$. Notice that $|{\rm Hess}y^j|\ls C(k_0)$ for some constant $C(k_0)$ depending on $k_0$, a  bound of $|sec_M|$ on $U$. This implies that
  $$|\phi|+|\nabla\phi|+|{\rm Hess}(\phi)|\ls \delta\cdot C(n,k_0),\qquad \forall\ x\in B_{r_0}(\hat x).$$
The proof is finished.
\end{proof}

\end{document}